\documentclass[aap,MSNbibl,dvips]{arximspdf}
\usepackage{mathrsfs}
\usepackage{graphicx}
\usepackage{url, breakurl}

\doi{10.1214/13-AAP978} 
\volume{24}
\issue{6}
\pubyear{2014}
\firstpage{2297}
\lastpage{2339}

\makeatletter
\newcommand{\widebar}{\overline}
\newcommand{\ddd}{d}
\newcommand{\rright}{\right}
\newcommand{\lleft}{\left}
\newcommand{\rrvert}{\vert}
\newcommand{\llvert}{\vert}
\newtheorem{teo}{Theorem}[section]
\newtheorem{lem}[teo]{Lemma}
\newtheorem{prop}[teo]{Proposition}
\newtheorem{cor}[teo]{Corollary}
\newproclaim{rem}{Remarks}
\newproclaim{re}{Remark}
\newproclaim{pop}{Plan of the paper}
\newproclaim{db}{Description of the backbone}
\newproclaim{aaa}{The pruned subtrees, and their grafting on the backbone}
\newcommand{\V}[1]{{\operatorname{Var}}[#1]}
\newcommand{\llbracket}{[\![}
\newcommand{\rrbracket}{]\!]}
\makeatother
\begin{document}
\begin{frontmatter}

\title{Cutting down trees with a Markov chainsaw}
\runtitle{Cutting down trees with a Markov chainsaw}

\begin{aug}
\author[A]{\fnms{Louigi} \snm{Addario-Berry}\ead[label=e1]{louigi@gmail.com}},
\author[B]{\fnms{Nicolas} \snm{Broutin}\corref{}\ead[label=e2]{nicolas.broutin@inria.fr}}
\and\break 
\author[C]{\fnms{Cecilia} \snm{Holmgren}\ead[label=e3]{cecilia@math.su.se}}
\runauthor{L. Addario-Berry, N. Broutin and C. Holmgren}
\affiliation{McGill University,
Inria Paris-Rocquencourt and
Stockholm University}
\address[A]{L. Addario-Berry\\
Department of Mathematics and Statistics\\
McGill University\\
805 Sherbrooke Street West\\
Montreal, Canada, H3A 0B9\\
\printead{e1}} 
\address[B]{N. Broutin\\
Inria Paris--Rocquencourt\\
Domaine de Voluceau\\
78153 Le Chesnay\\
France\\
\printead{e2}}
\address[C]{C. Holmgren\\
Department of Mathematics\\
Stockholm University\\
114 18 Stockholm\\
Sweden\\
\printead{e3}}
\end{aug}

\received{\smonth{5} \syear{2013}}
\revised{\smonth{10} \syear{2013}}

%
\begin{abstract}
We provide simplified proofs for the asymptotic distribution of the
number of cuts required to cut down a Galton--Watson tree with
critical, finite-variance offspring distribution, conditioned to have
total progeny $n$. Our proof is based on a coupling which yields a
precise, nonasymptotic distributional result for the case of uniformly
random rooted labeled trees (or, equivalently, Poisson Galton--Watson
trees conditioned on their size). Our approach also provides a new,
random reversible transformation between Brownian excursion and
Brownian bridge.
\end{abstract}

%

\end{frontmatter}

\section{Introduction}\label{secintro}

The subject of cutting down trees was introduced by Meir and Moon \cite
{MeMo1970,MeMo1974}. One is given a rooted tree $T$ which is pruned by
random removal of edges. At each step, only the portion containing the
root is retained (we refer to the portions not containing the root as
the \emph{pruned} portions), and the process continues until eventually
the root has been isolated. The main parameter of interest is the
random number of cuts necessary to isolate the root. The dual problem
of isolating a leaf or a node with a specific label has been considered
by Kuba and Panholzer \cite{KuPa2008a,KuPa2008}.

The procedure has been studied on different deterministic and random
trees. Essentially two kinds of random models have been considered for
the tree: \emph{recursive trees} with typical inter-node distances of
order $\log n$ \cite{MeMo1978,IkMo2007,DrIkMoRo2009,Holmgren2008} and
trees arising from critical, finite-variance branching processes
conditioned to have size $n$, with typical distances of order $\sqrt n$
\cite{Janson2006c,FiKaPa2006,Panholzer2003,Janson2004,Panholzer2006}.
In this paper, we are interested in the latter family, and will refer
to such trees as \textit{conditioned trees} for short.

For conditioned trees emerging from a progeny distribution with
variance $\sigma^2\in(0,\infty)$, once divided by $\sigma\sqrt n$, the
number of cuts required to isolate the root of a conditioned tree
of\vadjust{\goodbreak}
size $n$ converges in distribution to a Rayleigh random variable with
density $xe^{-x^2/2}$ on $[0,\infty)$. In this form, under only a
second moment assumption, this was proved by Janson \cite{Janson2006c};
below we discuss earlier, partial results in this direction. The fact
that the Rayleigh distribution appears here with a $\sqrt n$ scaling in
a setting involving conditioned trees struck us as deserving of
explanation. The Rayleigh distribution also arises as the limiting
distribution of the length of a path between two uniformly random nodes
in a conditioned tree, after appropriate rescaling.

In this paper we show that the existence of a Rayleigh limit in both
cases is not fortuitous. We will prove using a coupling method that the
number of cuts and the distance between two random vertices are
asymptotically equal in distribution (modulo a constant factor $\sigma
^2$). This approach yields as a by-product very simple proofs of the
results concerning the distribution of the number of cuts obtained in
\cite{Panholzer2003,FiKaPa2006,Janson2004,Janson2006c}; this is
explained in Section~\ref{secsigma}.

At the heart of our approach is a coupling which yields the \textit{exact}
distribution of the number of cuts for every fixed $n$, for the special
case of uniform Cayley trees (uniformly random labeled rooted trees).
Given a rooted tree $t$ and a sequence $S=(v_1,\ldots,v_k)$ of not
necessarily distinct nodes of $t$, consider an edge-removal procedure
defined as follows.
The \textit{planting of $t$ at $S$}, denoted $t\langle S \rangle$, is obtained
from $t$ by creating a new node $w_i$ for each $1 \le i \le k$, whose
only neighbor is $v_i$.
(If the $v_i$'s are not all distinct, then the procedure results in
multiple new vertices being connected to the same original vertex; if
$v_i=v_j$ for $i\ne j$, then $w_i\ne w_j$ are both connected to $v_i=v_j$.)
Let $W=\{w_1,\ldots, w_k\}$ be the set of new vertices (it may be more
natural to take $W$ as a sequence, since $S$ is a sequence, but taking
$W$ as a set turns out to be notationally more convenient later).
For a subgraph $t'$ of $t\langle S \rangle$ and a vertex $v$, we write
$C(v,t')$ for the connected component of $t'$ containing $v$; let also
$C(V,t')$ be the (minimal) set of connected components containing all
the vertices in a set $V$.

Let $F^{(0)}=t\langle S \rangle$, and for $j \ge0$, let $F^{(j+1)}$ be
obtained from $F^{(j)}$ by removing a uniformly random edge from among
all edges of $C(W, F^{(j)})$, if there are any such edges. The
procedure stops at the first time $j$ at which $C(W,F^{(j)})$ simply
consists of the set of new vertices $\{w_1,\ldots, w_k\}$. We call this
procedure \textit{planted cutting of $S$ in $t$}. We remark that Janson
\cite{Janson2004} already introduced the planted cutting procedure in
the case $k=1$. Note that if $t$ is a rooted tree with root $r$, then
$t\langle\{r\} \rangle$ contains only one node which is not a node of $t$,
and in this case the cutting procedure is almost identical to that
described in the first paragraph of the \hyperref[secintro]{Introduction}; see, however, the
remark just before Theorem~\ref{teokey}. Write $M=M(t,S)$ for the
(random) total number of edges removed in the above procedure. We
remark that for each $0 \le i \le M$, $F^{(i)}$ has $i+1$ connected
components, each of which is a tree.

%
\begin{teo}\label{teokcoup}
Fix $n \ge1$ and $k \ge1$, let $T_n$ be a uniform Cayley tree on
nodes $[n]=\{1,\ldots,n\}$, let $V_1,\ldots,V_k$
be independent, uniformly random nodes of $T_n$ and write $S_k =
(V_1,\ldots,V_k)$. Then
$M(T_n,S_k) - k$ is distributed as the number of edges spanned by the
root plus $k$ independent, uniformly random nodes in a uniform Cayley
tree of size $n$.
\end{teo}

For $k\ge1$, let $\chi_k$ be a chi random variable with $2k$ degrees
of freedom; the distribution of $\chi_k$ is given by
\[
{\mathbf P}(\chi_k \le x)=\int_0^x
\frac{2^{1-k}s^{2k-1}e^{-s^2/2}}{(k-1)!}\,\ddd s.
\]

%
\begin{cor}\label{corbert}
For any fixed $k$, as $n \to\infty$, $M(T_n,S_k)/\sqrt{n}$ converges
to~$\chi_k$ in distribution.
\end{cor}
The fact that, after rescaling, the number of edges spanned by the root
and $k$ random vertices in $T_n$ converges to $\chi_k$ in distribution
is well known; see, for example, Aldous \cite{Aldous1993a}, Lemma~21.
In Appendix~\ref{appendixcor} we sketch one possible proof of
Corollary~\ref{corbert} and briefly discuss stronger forms of convergence.

\begin{rem*}
\begin{longlist}[$\star$]
\item[$\star$] In the special case $k=1$, Theorem~\ref{teokcoup} states that
the number of edges required to isolate the planted node in a planted
uniform Cayley tree of size $n$ is identical in distribution to the
number of vertices on the path between two uniformly random nodes in a
uniform Cayley tree of size $n$. For the case $k=1$,
Chassaing and Marchand \cite{ChMa2009} have also announced a simple
bijective proof of this result, based on linear probing hashing.

\item[$\star$] After the current results were announced \cite{webref2}, and
independently of
our results, Bertoin \cite{Bertoin2012a} used powerful recent results of
Haas and Miermont \cite{HaMi2012a} to establish the distributional
convergence in Corollary~\ref{corbert}. Bertoin's results give a
different explicit interpretation of the number of cuts as the
asymptotic distance between two nodes. Bertoin and Miermont \cite
{BeMi2012a} also study the genealogy of the
fragmentation resulting from the removal of edges in a random order.

\item[$\star$] The original analyses by Meir and Moon \cite{MeMo1970} include
asymptotics for the mean and variance of the number of cuts. In recent
years, the subject of distributional asymptotics has been revisited by
several researchers. Panholzer \cite{Panholzer2003} and Fill, Kapur and
Panholzer \cite{FiKaPa2006} have studied the somewhat simpler case
where, the laws of the trees (as $n$ varies), satisfy a certain
consistency relation. More precisely, if $\mu_n$ is the law of the
$n$-vertex tree, the consistency condition requires that after one step
of the cutting procedure, conditional on the size $k$ of the pruned
fragment, the pruned fragment and the remaining tree are independent,
with respective laws $\mu_k$ and $\mu_{n-k}$. The class of random trees
which satisfy this property includes uniform Cayley trees. For this
class, they obtained the limiting distribution of various functionals
of the number of cuts using the method of moments, and gave an analytic
treatment of the recursive equation describing the cutting procedure.
Janson \cite{Janson2006c,Janson2004} used a representation of the
number of cuts in terms of generalized records in a labeled tree to
extend some of these results to all the family trees of critical
branching processes with offspring distribution having a finite
variance. His method is also based on the calculation of moments.
\end{longlist}
\end{rem*}

In the case $k=1$, our coupling approach also allows us to describe the
joint distribution of the sequence of pruned trees. In this paper, a
\textit{forest} is a sequence of rooted labeled trees $\mathbf
{f}=(t_1,\ldots,t_j)$ with pairwise disjoint sets of labels. In the
notation of
Theorem~\ref{teokcoup} and of the paragraph which precedes it, write
$M = M(T_n,S_1)$ and write $(T^{(1)},\ldots,T^{(M)})$ for the connected
components of $F^{(M)}$, listed in the order they are created during
the edge-removal\vspace*{1pt} procedure on $T_n\langle S_1 \rangle$. Note that the edge-removal procedure stops at the first time that $w_1$ is isolated, so
necessarily $T^{(M)}$ consists simply of the single vertex $w_1$. For
each $1 \le i \le M$, $T^{(i)}$ is a tree, which we view as rooted at
whichever node of $T^{(i)}$ was closest to $w_1$ in $T_n\langle S_1
\rangle$;
in particular, necessarily $T^{(M-1)}$ is rooted at~$V_1$.
%
\begin{teo}\label{teouniforest}
The forest $(T^{(1)},\ldots,T^{(M-1)})$ is distributed as a uniformly
random forest on $[n]$.
\end{teo}
The analysis which leads to Theorem~\ref{teouniforest} will also yield
as a by-product the following result.
%
\begin{teo}\label{teoreverse}
Let $F^n=(T_1,\ldots,T_{\kappa})$ be a uniformly random forest on
$[n]$. For each $i \in[\kappa-1]$, add an edge from the root of $T_i$
to a uniformly random node from among all nodes in $T_{i+1},\ldots,T_{\kappa}$. Call the resulting tree $T$, and view $T$ as rooted at
the root of $T_{\kappa}$. Then $T$ is distributed as a uniform Cayley
tree on $[n]$.
\end{teo}

It turns out that our coupling approach allows us to prove results
about a natural ``continuum version'' of the random cutting procedure
which takes place on the Brownian continuum random tree (CRT). Our main
result about randomly cutting the CRT is Theorem~\ref{teoexcbr},
below. Although we work principally in the language of $\mathbb{R}$-trees,
Theorem~\ref{teoexcbr} can be viewed as a new, invertible random
transformation between Brownian excursion and reflecting Brownian
bridge. Though the precise statement requires a fair amount of set-up,
if this set-up is taken for granted the result can be easily described.
(For the reader for whom the following three paragraphs are opaque, all
the below terminology will be re-introduced and formally defined later
in the paper.)\vadjust{\goodbreak}

Let $(\mathcal{T},d)$ be a CRT with root $\rho$ and mass measure $\mu$,
write $\operatorname{skel}(\mathcal{T})$ for its skeleton, and let
$\mathcal{P}$
be a
homogeneous Poisson point process on $\operatorname{skel}(\mathcal
{T}) \times
[0,\infty)$
with intensity measure $\ell\otimes\ddd t$, where $\ell$ is
the length
measure on the skeleton. We think of the second coordinate as a time
parameter. View each point $(p,\tau)$ of $\mathcal{P}$ as a potential
cut, but only make a cut at $p$ if no previous cut has fallen on the
path from the root $\rho$ to $p$. At each time $0 \le t < \infty$, this
yields a forest of countably many rooted $\mathbb{R}$-trees; we write
$\mathcal
{T}_t$ for the component of this forest containing $\rho$. Run to time
\textit{infinity}, this process again yields a countable collection of
rooted $\mathbb{R}$-trees, later called $(f_i, i \in I_{\infty})$.
Furthermore,
each element $f_i$ of the collection comes equipped with a time index
$\tau_i$ (the time at which it was cut).

For $0 \le t < \infty$, let $L(t)=\int_0^t \mu(\mathcal{T}_s)\,
\ddd s$,
and let $L(\infty)=\lim_{t \to\infty} L(t)$. It turns out that
$L(\infty)$ is almost surely finite.
Next, create a single compact $\mathbb{R}$-tree $(\mathcal{T}',d')$
from the
collection $(f_i,i \in I_{\infty})$ and the closed interval
$[0,L(\infty
)]$ by identifying the root of $f_i$ with the point $L(\tau_i) \in
[0,L(\infty)]$, for each $i \in I_{\infty}$, then taking the completion
of the resulting object.
Let $\mu'$ be the push-forward of $\mu$ under the transformation
described above.
%
\begin{teo}\label{teoforwardtrans}
The triples $(\mathcal{T}',d',\mu')$ and $(\mathcal{T},d,\mu)$ have the
same distribution.
Furthermore, $0 \in\mathcal{T}'$ and $L(\infty) \in\mathcal{T}'$ are
independent and
both have law~$\mu'$.
\end{teo}
Using the standard encoding of the CRT by a Brownian excursion,
we may take the triple $(\mathcal{T},d,\mu)$, together with the point
$\rho$, to be encoded by a Brownian excursion. Similarly,
it is possible to view the triple $(\mathcal{T}',d',\mu')$, together
with the points $0$ and $L(\infty)$,
as encoded by a
reflecting Brownian bridge; see Section~10 of~\cite{AlPi1994} (this is
also closely related to the ``forest floor'' picture of \cite
{bertoinpitman94path}).
From this perspective, the transformation from $(\mathcal{T},\rho)$ to
$(\mathcal{T}',0,L_{\infty})$ becomes a new, random transformation from
Brownian excursion to reflecting Brownian bridge.
When expressed in the language of Brownian excursions and bridges, this
theorem and our ``inverse transformation'' result, Theorem~\ref
{teoctsreverse}, below, have intriguing similarities to results from
Aldous and Pitman \cite{AlPi1994}; we briefly discuss this in
Appendix~\ref{secapp}.

As an immediate consequence of the above development, we will obtain
the following result.
Let $\nu(t)$ be the mass of the tagged fragment in the Aldous--Pitman
\cite{AlPi1994} fragmentation at time $t$. Then, $(\nu(t),t\ge0)$ is
distributed as $(\mu(\mathcal{T}_t), t\ge0)$ and we have the following.
%
\begin{cor}\label{corrayleigh}
The random variable $\int_0^{\infty} \nu(t)\,\ddd t$ has the 
standard\break
Rayleigh distribution.
\end{cor}
%
A different proof of this fact appears in a recent preprint by Abraham
and Delmas \cite{AbDe2013b}.
We also note that the identity in Theorem~\ref{teoforwardtrans} has
been generalized to the case of L\'evy trees in \cite{AbDe2013a}.

We are also able to explicitly describe the inverse of the
transformation of Theorem~\ref{teoforwardtrans},
and we now do so.
Let $(\mathcal{T},d,\mu)$ be a measured CRT, and let $\rho,\rho'$ be
independent random points in $\mathcal{T}$ with law $\mu$. Let $B$ be
the set of branch points of $\mathcal{T}$ on the path from $\rho$ to
$\rho'$.
For each $b \in B$ let
$\mathcal{T}_b$ be the set of points $x \in\mathcal{T}$ for which the
path from $x$ to $\rho$ contains a point $b' \in B$ with
$d(\rho,b') > d(\rho,b)$. In words, $\mathcal{T}_b$ is the set of
points in subtrees that ``branch off the path from $\rho$ to $\rho'$
after $b$.''
Then, independently for each point $b \in B$, let $y_b$ be a random
element of~$\mathcal{T}_b$, with law $\mu/\mu(\mathcal{T}_b)$.
Delete all nonbranch points on the path between $\rho$ and $\rho'$;
then, for each $b \in B$, identify the points $b$ and $y_b$.
Write $(\mathcal{T}',d')$ for the resulting tree, and $\mu'$ for the
push-forward of $\mu$ to~$\mathcal{T}'$.
%
\begin{teo}\label{teoctsreverse}
The triples $(\mathcal{T},d,\mu)$ and $(\mathcal{T}',d',\mu')$ have the
same distribution.
Furthermore, the point $\rho' \in\mathcal{T}'$ has law $\mu'$.
\end{teo}
We remark that it is not \textit{a priori} obvious the inverse
transformation should a.s. yield a connected metric space, let alone
what the distribution of the resulting space should be.
Theorems~\ref{teoforwardtrans} and~\ref{teoctsreverse} together
appear as Theorem~\ref{teoexcbr}, below.

\begin{pop*}
In Section~\ref{secnotation} we
gather definitions and state our notational conventions. In
Section~\ref{seccayley} we prove all finite distributional identities
related to
the case $k=1$, in particular proving Theorems~\ref{teouniforest}
and~\ref{teoreverse}, and in Section~\ref{secmorethanone} we\vadjust{\goodbreak} prove
Theorem~\ref{teokcoup}. Our results on cutting the CRT, notably
Theorem~\ref{teoexcbr}, appear in Section~\ref{secnovel}; finally, in
Section~\ref{secsigma} we explain how our results straightforwardly
imply the distributional convergence results obtained in \cite
{Panholzer2003,Janson2004,Janson2006c}.
\end{pop*}

\section{Notation and definitions}\label{secnotation}
We note that the terminology introduced in Sections~\ref{secrealtrees}
and~\ref{sectypos} is not used until Section~\ref{secnovel}, and the
reader may wish to correspondingly postpone their reading of these sections.

\subsection{Finite trees and graphs}
Given any finite graph $G$, we write $v(G)$ for the set of vertices (or
\textit{nodes}) of $G$ and $e(G)$ for the set of edges of $G$, and write
$|G|$ for the size (number of vertices) of $G$.
If we say that \textit{$G$ is a graph on $S$}, we mean that $v(G)=S$.
Given a graph $G$ and $w \in v(G)$, we write $C(w,G)$ for the connected
component of $G$ containing
$w$. Given a graph $G$ and $S' \subset e(G)$, we sometimes write
$G\setminus S'$ for the graph
$(v(G),e(G)\setminus S')$.

Practically all graphs in this paper will be rooted trees and be
denoted $t$ or $T$. When we write ``tree'' we mean a rooted tree unless
we explicitly say otherwise.

Given a rooted labeled tree $t$, we write $r(t)$ for the root of $t$.
For a vertex $u$ of $t$ write $t(u)$ for the subtree of $t$ rooted at
$u$, write $h_t(u)$ for the number of edges on the path from $r(t)$ to
$u$, and write $a(u)=a(u,t)$ for the parent of $u$ in $t$, with the
convention that $a(r(t))=r(t)$.
At times we view the edges of $t$ as oriented toward $r(t)$. In other
words, if we state that $(u,v)$ is an oriented edge of $t$, or write
$(u,v) \in e(t)$, we mean that $\{u,v\} \in e(t)$ and $v=a(u)$. In this
case we call $u$ the \textit{tail} of $\{u,v\}$ and $v$ the \emph
{head} of
$\{u,v\}$.
It is also sometimes useful to view $r(t)$ as both the head and tail of
a directed loop $(r(t),r(t))$; we will mention this again when it arises.

Given a set $S=\{v_1,\ldots,v_k\}$ of nodes of $t$, we write
$t{\llbracket S \rrbracket}$
or $t{\llbracket v_1,\ldots,v_k \rrbracket}$ for the subtree of $t$
obtained by taking
the union of all shortest paths between elements of $S$, and call $t
{\llbracket S \rrbracket}$ \emph{the subtree of $t$ spanned by $S$};
if $r(t) \in S$ then
we consider $t{\llbracket S \rrbracket}$ as rooted at $r(t)$. Given a
single node $v \in
t$, we write $t^{r \leftrightarrow v}$ to denote the tree obtained from
$t$ by
rerooting at $v$. As mentioned in the \hyperref[secintro]{Introduction}, in this paper an
\emph{ordered forest} is a sequence of rooted labeled trees $\mathbf{f}
=(t_1,\ldots,t_k)$ with pairwise disjoint sets of labels. If we write
\textit{$\mathbf{f}=(t_1,\ldots,t_k)$ is an ordered forest on~$S$}
we mean that
$v(t_1) \cup\cdots\cup v(t_k) = S$.

Given a finite set $S$, by a \emph{uniform Cayley tree on $S$} we mean a
rooted tree chosen uniformly at random from among all rooted trees $t$
on $S$; there are $|S|^{|S|-1}$ such trees.
Given a rooted or unrooted tree $t$, and an ordered sequence
$S=(v_1,\ldots,v_k)$ of elements of $v(t)$,
we recall the definition of $t \langle S \rangle$ (the \emph{planting
of $t$ at
$S$}) from the \hyperref[secintro]{Introduction}: for each $1 \le i \le k$, create a new
node $w_i$ and add a single edge between $w_i$ and $v_i$.
Given a set $U \subset v(t\langle S \rangle)$, we write $|U|$ for the number
of nodes of $U \setminus\{w_1,\ldots,w_k\}$. In other words, the nodes
$w_1,\ldots,w_k$ \emph{are not included} when performing node counts in
$t\langle S \rangle$.

\subsection{Metric spaces and real trees}\label{secrealtrees}
In this paper all metric spaces are assumed to be separable.
Given a metric space $\mathrm{X}=(X,d)$, and a real number $c>0$, we\vadjust{\goodbreak} write
$c\mathrm{X}$ for the metric space obtained by scaling all distances
by $c$.
In other words, if $x,y \in X$, then
the distance between $x$ and $y$ in $c\mathrm{X}$ is $cd(x,y)$. We
also write
$\operatorname{diam}(\mathrm{X}) = \sup\{d(x,y)\dvtx x,y \in X\} \in
[0,\infty]$.

Given a metric space $(X,d)$ and $x,y \in X$, a \emph{geodesic between
$x$ and $y$} is
an isometry $f\dvtx[0,d(x,y)] \to X$ such that $f(0)=x$ and $f(d(x,y))=y$.
In this case we call the image
$\operatorname{Im}(f)$ a \emph{shortest path between $x$ and $y$}.

A metric space $\mathrm{T}=(T,d)$ is an {$\mathbb{R}$-tree} if for
all $x,y \in T$ the
following two properties hold:
\begin{longlist}[(1)]
\item[(1)] There exists a unique geodesic between $x$ and $y$. In other
words, there exists a unique
isometry $f\dvtx[0,d(x,y)] \to T$ such that $f(0)=x$ and $f(d(x,y))=y$.
\item[(2)] If $g\dvtx[0,d(x,y)] \to T$ is a continuous injective map with
$g(0)=x$ and $g(d(x,y))=y$, then $f([0,d(x,y)])=g([0,d(x,y)])$.
\end{longlist}
Given an $\mathbb{R}$-tree $(T,d)$ and $a,b \in T$, we write
${\llbracket a,b
\rrbracket}$ for
the image of the unique geodesic from $a$ to $b$, and write $\rrbracket
a,b\llbracket= {\llbracket a,b \rrbracket}\setminus\{a,b\}$.
The \emph{skeleton} $\operatorname{skel}(\mathrm{T})$ is defined as
\[
\bigcup_{a,b \in T}\rrbracket a,b\llbracket.
\]
(We could equivalently define $\operatorname{skel}(\mathrm{T})$ as
the set
of points whose
removal disconnects the space.)
Since $(T,d)$ is separable by assumption, this may be re-written as a
countable union, and so there is a unique $\sigma$-finite measure
$\ell
$ on $T$ with $\ell(]a,b[) = d(a,b)$ for all $a,b \in T$ and such that
\mbox{$\ell(T\setminus\operatorname{skel}(T))=0$}. We refer to $\ell$ as
the \emph{length
measure} on $\mathrm{T}$.

For a set $S \subset T$, write $T{\llbracket S \rrbracket}$ for the
subspace of $T$
spanned by $\bigcup_{x,y \in S} \rrbracket x,y\llbracket$ and $d_S$ for
its distance (the restriction of $d$ to $T{\llbracket S \rrbracket}$),
and note that
$(T{\llbracket S \rrbracket},d_S)$ is again a real tree.

\subsection{Types of convergence} \label{sectypos}
Before proceeding to definitions, we remark that not all the
terminology of this subsection is yet fully standardized. The
Gromov--Hausdorff distance is by now well-established.
The name ``Gromov--Hausdorff--Prokhorov distance'' seems to have first
appeared in \cite{villani2009optimal}, Chapter~27, where it had a
slightly different meaning. The probabilistic aspects of the
Gromov--Hausdorff--Prokhorov distance were substantially developed in
\cite{miermont2009tessellations,HaMi2012a}. In particular, it is shown
in \cite{miermont2009tessellations}, Section~6.1, that the below
definition of $d_{\mathrm{GHP}}$ is equivalent to a definition based
on the more
standard Prokhorov distance between measures.

\subsubsection*{Gromov--Hausdorff distance}
Let $\mathrm{X}=(X,d_X)$ and $\mathrm{Y}=(Y,d_Y)$ be compact metric
spaces. The
\emph{Gromov--Hausdorff distance} $d_{\mathrm{GH}}(\mathrm
{X},\mathrm{Y})$
between $\mathrm{X}$ and $\mathrm{Y}$ is
defined as follows. Let $\mathcal{S}$ be the set of all pairs $(\phi,\psi)$,
where $\phi\dvtx X \to Z$ and $\psi\dvtx Y \to Z$ are isometric
embeddings into
some common metric space $(Z,d_Z)$. Then
\[
d_{\mathrm{GH}}(\mathrm{X},\mathrm{Y}) = \inf_{(\phi,\psi) \in
\mathcal{S}}
d_{\mathrm{H}} \bigl(\phi(X), \psi(Y) \bigr),
\]
where $d_{\mathrm{H}}$ denotes Hausdorff distance in the target metric space.
It can be verified that $d_{\mathrm{GH}}$ is indeed a distance and
that, writing
$\mathcal{M}$ for the set of isometry-equivalence classes of compact metric
spaces, $(\mathcal{M},d_{\mathrm{GH}})$ is a complete separable
metric space. We say that
a sequence $\mathrm{X}_n=(X_n,d_n)$ of compact metric spaces converges
to a
compact metric space $\mathrm{X}=(X,d)$ if $d_{\mathrm{GH}}(\mathrm
{X}_n,\mathrm{X}) \to
0$ as $n \to
\infty$. It is then obvious that $X$ is uniquely determined up to
isometry. There are two alternate descriptions of the Gromov--Hausdorff
distance that will be useful and which we now describe.

Next, for compact metric spaces $(X,d_X)$ and $(Y,d_Y)$, and a subset
$C$ of $X\times Y$,
the \emph{distortion} $\operatorname{dis}(C)$ is defined by
\[
\operatorname{dis}(C) = \sup \bigl\{\bigl|d_X \bigl(x,x'
\bigr)-d_Y \bigl(y,y' \bigr)\bigr|\dvtx(x,y) \in C,
\bigl(x',y' \bigr) \in C \bigr\}.
\]
A \emph{correspondence} $C$ between $X$ and $Y$ is a Borel subset of $X
\times Y$ such that for every $x \in X$, there exists $y \in Y$ with
$(x,y) \in C$ and vice versa. Write $\mathscr{C}(X,Y)$ for the set of
correspondences between $X$ and $Y$.
We then have
\[
d_{\mathrm{GH}}(\mathrm{X},\mathrm{Y}) = \tfrac{1}{2}\inf \bigl\{ r\dvtx
\exists C \in\mathscr{C}(X,Y) \mbox{ such that } \operatorname{dis}(C) < r
\bigr\}
\]
and there is a correspondence which achieves this infimum.

Given a correspondence $C$ between $X$ and $Y$
and $\varepsilon\ge0$ write
\[
C_{\varepsilon} = \bigl\{(x,y) \in X \times Y\dvtx\exists \bigl(x',y'
\bigr) \in C, d_X \bigl(x,x' \bigr) \le\varepsilon,
d_{Y} \bigl(y,y' \bigr) \le\varepsilon \bigr\}
\]
and note that $C_{\varepsilon}$ is again a correspondence, with
distortion at
most $\operatorname{dis}(C) + 2\varepsilon$.
We call $C_{\varepsilon}$ the \textit{$\varepsilon$ blow-up of $C$}.

Let $\mathrm{X}=(X,d_X,(x_1,\ldots,x_k))$ and $\mathrm
{Y}=(Y,d_Y,(y_1,\ldots,y_k))$
be metric spaces, each with an ordered set of $k$ distinguished points
(we call such spaces \textit{\mbox{$k$-}pointed metric spaces}). When $k=1$, we
simply refer to pointed (rather than $1$-pointed) metric spaces, and
write $(X,d_X,x)$ rather than $(X,d_X,(x))$.
The \textit{$k$-pointed Gromov--Hausdorff distance} is defined as
\begin{eqnarray*}
&& d_{\mathrm{GH}}^k(\mathrm{X},\mathrm{Y})
\\
&&\quad = \tfrac{1}{2}\inf
\bigl\{ r\dvtx \exists C \in \mathscr{C}(X,Y)\mbox{ such that }
(x_i,y_i) \in C, 1 \leq i\leq k\mbox{ and }
\operatorname{dis}(C) < r \bigr\}.
\end{eqnarray*}
It is straightforward to verify that for each $k$, the space $(\mathcal{M}
^k,d_{\mathrm{GH}}^k)$ of marked isometry-equivalence classes of $k$-pointed
compact metric spaces, endowed with the distance $d_{\mathrm{GH}}^k$,
forms a
complete separable metric space.

\subsubsection*{Couplings and Gromov--Hausdorff--Prokhorov distance}
Let $(X,d,\mu)$ and $(X',d',\mu')$ be two measured metric spaces, and
let $\nu$ be a
Borel measure on $X \times X'$.
We say $\nu$ is a (\textit{defective}) \textit{coupling} between
$\mu$ and $\mu'$
if $p_*\nu\le\mu$ and $p'_*\nu\le\mu'$, where
$p\dvtx X\times X'\to X$ and $p'\dvtx X\times X'\to X'$ are the canonical
projections.
The \emph{defect} of $\nu$ is defined as
\[
D(\nu)=\max \bigl((\mu-p_*\nu) (X), \bigl(\mu'-p'_*
\nu \bigr) \bigl(X' \bigr) \bigr).
\]
We let $\mathcal{C}(\mu,\mu')$ be the set of couplings between $\mu$
and $\mu'$, and for $\varepsilon\ge0$ we write $\mathcal{C}_{\varepsilon
}(\mu,\mu')
= \{\nu\in\mathcal{C}(\mu,\mu')\dvtx D(\nu) \le\varepsilon\}$

The \emph{Prokhorov distance} between two finite positive
Borel measures $\mu,\mu'$ on the same space $(X,d)$ is
\begin{eqnarray*}
\ddd^\circ_{\mathrm{P}} \bigl(\mu,\mu' \bigr)
&=& \inf \bigl\{\varepsilon>0\dvtx\mu(F)\leq\mu' \bigl(F^\varepsilon
\bigr)+ \varepsilon\mbox{ and }\mu'(F)\leq\mu
\bigl(F^\varepsilon \bigr)+\varepsilon
\\
&&\hspace*{136pt} \mbox{ for every closed }F\subseteq X\bigr\},
\end{eqnarray*}
where $F^\varepsilon=\{x\in X\dvtx\exists x' \in F, d(x,x')<\varepsilon\}$.

There is another distance which generates the same topology and lends
itself more naturally to combination with the correspondences
introduced above.
We define
\[
\ddd_{\mathrm{P}} \bigl(\mu,\mu' \bigr)=\inf \bigl\{
\varepsilon>0\dvtx\exists\nu\in\mathcal{C}_{\varepsilon} \bigl(\mu,
\mu' \bigr),\nu \bigl( \bigl\{ \bigl(x,x'
\bigr) \in X \times X\dvtx d \bigl(x,x' \bigr)\geq\varepsilon \bigr\}
\bigr)< \varepsilon \bigr\}.
\]
By analogy with the latter, the \emph{Gromov--Hausdorff--Prokhorov}
(GHP) distance between $\mathrm{X}=(X,d,\mu)$ and $\mathrm
{X}'=(X',d',\mu')$ is
defined as
\[
d_{\mathrm{GHP}} \bigl(\mathrm{X},\mathrm{X}' \bigr) = \inf\lleft\{ \varepsilon > 0\dvtx\quad \matrix{ \exists\nu\in\mathcal{C}_\varepsilon
\bigl(\mu, \mu' \bigr)\mbox{ and } R \in\mathscr{C} \bigl(X,
X' \bigr)\mbox{ such that}
\vspace*{3pt}\cr
\nu \bigl(R^c
\bigr) < \varepsilon, \operatorname{dis}(R) < 2 \varepsilon} \rright\}.
\]
We always have $d_{\mathrm{GHP}}(\mathrm{X},\mathrm{X}') \geq
\ddd_{\mathrm
{GH}}(\mathrm{X},\mathrm{X}')$.
Similarly to before, the collection $\widehat{\mathcal{M}}$ of measured
isometry-equivalence classes of compact metric spaces, endowed with the
distance $d_{\mathrm{GHP}}$, forms a complete separable metric space
\cite{miermont2009tessellations}, Section~6.

Given $\mathrm{X}=(X,d_X,\mu,(x_1,\ldots,x_k))$ and $\mathrm
{X}'=(X',d',\mu
',(x_1',\ldots,x_k'))$,
two \mbox{$k$-}pointed measured metric spaces, we define
the \textit{$k$-pointed Gromov--Hausdorff--Prokhorov distance} as
\begin{eqnarray*}
&& d_{\mathrm{GHP}}^k \bigl(\mathrm{X},\mathrm{X}' \bigr)
\\
&&\qquad = \inf \lleft\{ \varepsilon> 0\dvtx\quad\matrix{ \exists\nu\in
\mathcal{C}_{\varepsilon
} \bigl(\mu, \mu' \bigr)\mbox{ and } R \in\mathscr{C} \bigl(X, X' \bigr)\mbox{ such that}
\vspace*{3pt}\cr
\nu \bigl(R^c \bigr) < \varepsilon, \operatorname{dis}(R) < 2\varepsilon
\mbox{ and } \bigl(x_i,x_i' \bigr)
\in R, 1 \leq i \leq k} \rright\}.
\end{eqnarray*}
Once again, we may define an associated complete separable metric space
$(\widehat{\mathcal{M}}^k,d_{\mathrm{GHP}}^k)$.


\section{Cutting down uniform Cayley trees}\label{seccayley}
\subsection{The Aldous--Broder dynamics}
Given a simple random walk $\{X_n\}_{n \in\mathbb N}$ on a finite connected
graph $G$, we may generate a spanning tree $T$ of $G$ by including all
edges $(X_k,X_{k+1})$ with the property that $X_{k+1} \notin\{X_i\}_{0
\leq i \leq k}$. The resulting tree $T$ is in fact almost surely a
uniformly random spanning tree of $G$. (More generally, if $G$ comes
equipped with edge weights $\{w_e\dvtx e \in e(G)\}$, then the probability
the simple random walk on the weighted graph $G$ generates a specific
spanning tree $t$ is proportional to $\prod_{e \in e(t)} w_e$.) This
fact was independently discovered by Broder \cite{Broder1989} and
Aldous \cite{Aldous1990}, and the above procedure is commonly called
the Aldous--Broder algorithm.

By reversibility, the tree $T$ generated by the Aldous--Broder
algorithm may instead be viewed as generated by a simple random walk $\{
X_n\}_{n \le0}$ on $G$, started from stationarity at time $-\infty$;
see \cite{lp}, pages 127--128. If instead of stopping the walk at time
zero we instead stop at time $i\ge0$, then the walk $\{X_n\}_{n \le
i}$ gives another tree, say $T_i$. What we call the \emph{Aldous--Broder
dynamics} is the (deterministic) rule by which the sequence $\{T_i,\break i
\ge0\}$ is obtained from $T_0$ and from the sequence $\{X_n,n \ge0\}
$. In the current section, we explain these dynamics. In the next
section, we introduce a modification of the Aldous--Broder dynamics,
and use it to exhibit the key coupling alluded to in Section~\ref{secintro}.

Recall that given a rooted tree $t$ and $x \in v(t)$, $t(x)$ denotes
the subtree of $t$ rooted at $x$. Fix an integer $n \geq1$ and a tree
$t$ on $[n]$, and let $\mathbf{x}= (x_i)_{i \in\mathbb N}$ be a
sequence of
elements of $[n]=\{1,2,\ldots,n\}$.

We then form a sequence of trees $\{T^m(t,\mathbf{x})\dvtx m \in
\mathbb N\}
$. First, $T^0=t$.
Then, for $m \ge0$, we proceed as follows:
\begin{itemize}
\item if $x_{m+1}=r(T^m)$, then $T^{m+1}=T^{m}$;
\item if $x_{m+1}\neq r(T^m)$, then form $T^{m+1}$ by removing the
unique edge of $T^m$ with tail $x_{m+1}$, then adding the edge
$(x_{m},x_{m+1})$, and finally rerooting at $x_{m+1}$.
\end{itemize}
In all cases, $r(T^m)=x_m$ for all $m \ge1$. We refer to this
procedure as \emph{the Aldous--Broder dynamics on $t$ and $\mathbf{x}$}.
One can equivalently think of the root vertex as being both the head
and tail of a directed loop; then one always removes the unique edge
with tail $x_{m+1}$ in $T^m$ and adds the directed edge
$(x_m,x_{m+1})$. Taking this perspective, let
$R_{m+1}=R_{m+1}(t,\mathbf{x})$
be the subtree of $T^m$ rooted at $x_{m+1}$, so $R_{m+1}=T^m(x_{m+1})$.
Let $K_{m+1}=K_{m+1}(t,\mathbf{x})$ be the other component created when
removing the edge with tail $x_{m+1}$, which is empty if $x_{m+1}=x_m$
and otherwise contains $x_{m}$. In all cases $T^{m+1}$ is obtained from
$R_{m+1}$ and $K_{m+1}$ by adding an edge from $x_m$ to $x_{m+1}$; see
Figure~\ref{figdoyle}.

%
\begin{figure}

\includegraphics{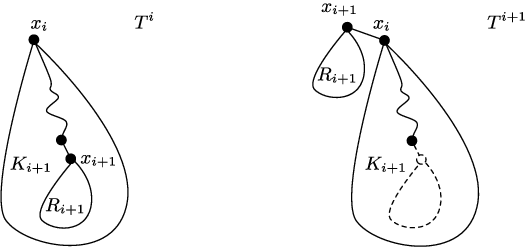}

%
%
\caption{Two successive trees $T^i$ and $T^{i+1}$ built from the
sequence construction: $T^{i+1}$ is obtained from $T^i$ by cutting
above $x_{i+1}$ and rearranging the parts in such a way that the
subtree above the cut is appended as a child of the root $x_{i+1}$ of
the subtree $R_{i+1}$ below the cut.}\label{figdoyle}
\end{figure}

\subsection{A modified Aldous--Broder dynamics}\label{sectrefor}
Say that a sequence $\mathbf{x}\in[n]^{\mathbb N}$ is \emph{good} if
for each $k \in[n]$, $\sup\{i\dvtx x_i=k\}=\infty$.
Fix a tree $t$ on $[n]$ and a good sequence $\mathbf{x}$.
We now describe a rule for removing a set of edges from $t$ to obtain
an ordered forest $\mathbf F=\mathbf F(t,\mathbf{x})$ on $[n]$.
[Recall that an \emph{ordered forest} is an ordered sequence
$(t_1,\ldots,t_k)$ of rooted trees.]

%
\begin{figure}

\includegraphics{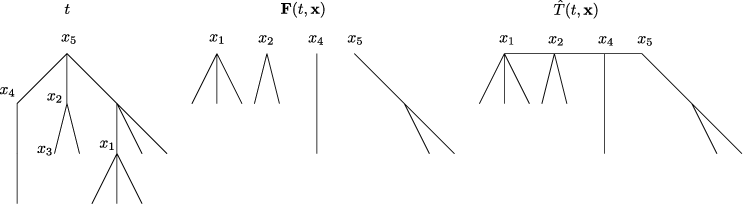}

\caption{Left: a tree $t$, with node labels suppressed for readability;
the first five nodes $x_1,\ldots,x_5$ of some good sequence are marked
in the figure. Center: the forest $\mathbf{F}(t,\mathbf{x})$ built by
applying the modified Aldous--Broder dynamics to $t$ with any sequence
$\mathbf{x}$ starting with $x_1,\ldots,x_5$. The trees are
$T_1(t,\mathbf{x}),\ldots,T_4(t,\mathbf{x})$ are shown from left to
right, and $r_1=x_1,r_2=x_2,r_3=x_4,r_4=x_5$. Right: the tree $\widehat
{T}(t,\mathbf{x})$, which has root $x_1$.}
\label{figmoddyn}
\end{figure}

In words, to build $\mathbf F(t,\mathbf{x})$ we start from the tree
$t$ and make the
cuts that are dictated by the sequence $\mathbf{x}$, but \emph
{ignore} any such
cuts that fall in a subtree we have already pruned at an earlier step.
Since $\mathbf{x}$ is good, we will eventually prune the root $r(t)$
and so we
will ignore all but finitely many of the cuts.

Formally, let $\sigma_0=0$ and, for $i \ge1$, let
\[
\sigma_i = \inf \Biggl\{m > \sigma_{i-1}\dvtx
x_m \notin\bigcup_{j=1}^{i-1}
t(x_{\sigma_j}) \Biggr\}.
\]
Then let $\kappa= \kappa(t,\mathbf{x}) = \inf\{i\dvtx x_{\sigma
_i}=r(t)\}$. Note
that we always have $\sigma_1=1$, that $\kappa< \infty$ since
$\mathbf{x}$ is
good, and that for all $j > \kappa$, $\sigma_j = \infty$.
Recall that we write $t=(v(t),e(t))$, where $v(t)$ and $e(t)$ denote
the vertex and edge set of $t$, respectively. After all the cuts in
$\mathbf{x}
$ have been made, we are left with a graph
\[
f = \bigl(v(t), e(t) \setminus \bigl\{ \bigl(x_{\sigma
_i},a(x_{\sigma_i})
\bigr),1 \le i \le\kappa \bigr\} \bigr).
\]
For $1 \le i \le\kappa$, let $T_i = T_i(t,\mathbf{x}) = C(x_{\sigma_i},f)$.
Note that $T_i$ is a tree, which we view as rooted at $x_{\sigma_i}$.
We then take
\[
\mathbf F= \mathbf F(t,\mathbf{x}) = (T_1,\ldots,T_\kappa).
\]
Write $r_i=r_i(t,\mathbf{x})$ for the root of $T_i$ and note that
$r_\kappa=r(t)$.
Finally, write $\widehat{T}=\widehat{T}(t,\mathbf{x})$ for the tree obtained
from the
forest $\mathbf F(t,\mathbf{x})$ by adding a directed edge from the
root of
$T_{i+1}$ to the root of $T_{i}$, for each $i \in[\kappa-1]$, and
rooted at $r_1$ (as suggested by the orientation of the edges). These
definitions are illustrated in Figure~\ref{figmoddyn}. We call this
procedure the \emph{modified Aldous--Broder dynamics on $t$ and
$\mathbf{x}$.}

\begin{re*}
The cutting procedure described above differs
slightly from that used in much of the work on the subject. More
precisely, it is more common to cut the tree by the removal of random
edges rather than the selection of random vertices. However, there is a
close correspondence between the vertex selection procedure and the
edge selection procedure on a planted version of the same tree, which
means results proved for one procedure have immediate analogues for the
other. In particular, Janson (\cite{Janson2004}, Lemma 6.1) analyzed
the difference between the two variants and showed that it is
asymptotically negligible.
\end{re*}

Now let $\mathcal{X}=(X_m)_{m \in\mathbb N}$ be a sequence of i.i.d.
uniform $\{
1,\ldots,n\}$ random variables.
It is easily seen that $\mathcal{X}$ is good with probability one.
The following theorem is the key fact underlying almost all the results
of the paper.
%
\begin{teo}\label{teokey}
Let $T$ be a uniform Cayley tree on $[n]$. Then for any tree $t$ on
$[n]$ and any $w \in[n]$,
\[
{\mathbf P}\bigl(\widehat{T}(T,\mathcal{X}) =t\mbox{ and } r(T)=w\bigr) =
n^{-n}.
\]
\end{teo}
Since there are $n^{n-1}$ labeled rooted trees on $[n]$, there are
$n^{n}$ possible ways to choose a labeled rooted tree on $[n]$, plus an\vspace*{1pt}
additional vertex of said tree. In other words, the theorem states that
$\widehat{T}(T,\mathcal{X})$ is a uniform Cayley tree, and that $r(T)$ is uniform
on $[n]$ and independent of $\widehat{T}(T,\mathcal{X})$
{(the fact that $r(T)$ is uniform on $[n]$ is immediate from the fact
that $T$ is a uniform Cayley tree).}
\begin{pf*}{Proof of Theorem~\ref{teokey}}
We proceed by induction on $n$, the case $n=1$ being trivial. So we now
suppose that $n>1$.
First, consider the case when \mbox{$w=r(\widehat T)$;} we have $r(T)=r(\widehat{T})$
precisely if $X_1=r(T)$ and in this case $\widehat{T}=T$. Thus, for any
rooted tree $t$ on $[n]$,
\[
{\mathbf P}\bigl(\widehat{T}=t,r(T)=r(\widehat{T})\bigr) = {\mathbf P}
\bigl(X_{1}=r(T),T=t\bigr) = \frac{1}{n}{\mathbf P}(T=t)=
\frac{1}{n^n},
\]
since $T$ is a uniform Cayley tree.

Next, fix a rooted tree $t$ on $[n]$ and any $w \in[n]$, $w \neq r(t)$.
Let $c=c(t,w)$ be the child of $r=r(t)$ for which the subtree of $t$
rooted at $c$ contains the node $w$. Let $t_r$ and $t_c$ be the
subtrees containing $r$ and $c$, respectively,\vspace*{1pt} when the edge $(c,r)$ is
removed from $t$.
If we are to have $r(T)=w$ and $\widehat{T}=t$, then $t_r$ must appear as a
subtree of $T$, and we must additionally have $X_{1}=r$. Since $T$ is a
uniform Cayley tree it follows that
%
\begin{eqnarray}\label{eqkey1}
&& {\mathbf P}\bigl(r(T) = w,\widehat{T}=t\bigr)\nonumber
\\
&&\qquad = {\mathbf P}\bigl(r(T)=w,\widehat{T}=t, t_r\mbox{ is a subtree of } T, X_{1}=r\bigr)
\\
&&\qquad = \frac{(n-|t_r|)^{n-|t_r|}}{n^{n-1}} \cdot\frac{1}{n} \cdot {\mathbf P}\bigl(r(T)=w,
\widehat{T}=t|t_r \mbox{ is a subtree of }T, X_{1}=r\bigr).\hspace*{-20pt}\nonumber
\end{eqnarray}
Now let $\mathcal{X}'=(X_i')_{i \in\mathbb N}$ be the subsequence of
$\mathcal{X}$ consisting
of the nodes of $K_1(T,\mathcal X)$, the connected component of $T$
containing the root after the edge above $X_1$ has been removed:
for $i \in\mathbb N$, let
\[
j_i=\min \bigl\{\ell\dvtx\bigl|\{X_{1},\ldots,X_{\ell}
\} \cap v \bigl(K_1(T,\mathcal{X}) \bigr)\bigr|=i \bigr\}
\]
and set $X_{i}'=X_{j_i}$.
Given that $t_r$ is a subtree of $T$ and $X_{1}=r$,
the entries of $\mathcal{X}'$ are independent, uniformly random
elements of
$v(t_c)$. Furthermore, under this conditioning we have that $\widehat
{T}(T,\mathcal{X})=t$ and $r(T)=w$ precisely if $\widehat
{T}(K_1(T,\mathcal{X}),\mathcal{X}')=t_c$
and $r(K_1(T,\mathcal{X}))=w$. Since $T$ is a uniform Cayley tree and
$K_1(T,\mathcal{X}
)$ is obtained from $T$ by removing the subtree rooted at $X_1$, it is
immediate that conditional on its vertex set, $K_1(T,\mathcal{X})$ is
again a
uniform Cayley tree (and has less vertices than~$T$). By induction, it
follows that
\begin{eqnarray*}
&& {\mathbf P}\bigl(r(T) = w, \widehat{T}=t|t_r \mbox{ is a subtree of }T, X_{1}=r\bigr)
\\
&&\qquad ={\mathbf P}\bigl(\widehat{T} \bigl(K_1(T,\mathcal{X}),
\mathcal{X}' \bigr)=t_c, r \bigl(K_1(T,
\mathcal{X}) \bigr)=w|t_r \mbox{ is a subtree of }T,
X_{1}=r\bigr)
\\
&&\qquad = |t_c|^{-|t_c|}.
\end{eqnarray*}
Since $|t_c|=n-|t_r|$, together with (\ref{eqkey1}) this yields that
${\mathbf P}(\widehat{T}(T,\mathcal{X}) =t$ and $r(T)=w) = n^{-n}$, as required.
\end{pf*}

We can transform the modified Aldous--Broder procedure for isolating
the root into an edge-removal procedure, as follows. First, plant the
tree to be cut at its root. Next, each time a node is selected for
pruning, instead remove the parent edge incident to each selected
vertex. The Aldous--Broder procedure then becomes the planted cutting
procedure described in the \hyperref[secintro]{Introduction}, and
$\kappa(T,\mathcal{X})$ is precisely the number of edges removed
before the
planted vertex is isolated.
But $\kappa(T,\mathcal{X})$ is also the number of vertices on the
path from
$r(\widehat{T})$ to $r(T)$ in $\widehat{T}$. By Theorem~\ref{teokey}, and from
known results about the distance between the root and a uniformly
random node in a uniform Cayley tree \cite
{MeMo1978,Kolchin1986,Aldous1991,Aldous1991b,Aldous1993a},
the case $k=1$ of Theorem~\ref{teokcoup} and of Corollary~\ref
{corbert} follow immediately.
By a well-known bijective correspondence between labeled rooted trees
with a distinguished vertex and ordered labeled rooted forests (see, e.g., \cite{AlPi1994}), Theorem~\ref{teouniforest}
also follows
immediately (the forest consists of the sequence of trees obtained when
removing the
edges on the path between the root and the distinguished vertex).

\begin{re*}
Aldous \cite{Aldous1991a} studied the
subtree rooted at a uniformly random node in a critical, finite
variance Galton--Watson tree conditioned to have size $n$. In
particular, he showed that such a subtree converges in distribution to
an \emph{unconditioned} critical Galton--Watson tree. It is then
straightforward that, for fixed $k\ge1$, the first $k$ trees that are
cut converge in distribution to a forest of $k$ critical Galton--Watson
trees. On the other hand, a critical Galton--Watson tree conditioned to
be large converges locally (in the sense of local weak convergence of
\cite{AlSt2003}, i.e., inside balls of arbitrary fixed radius $k$
around the root) to an infinite path of nodes having a size-biased
number of children (exactly one of which is again on the infinite
path), where each nonpath node is the root of an unconditioned
critical Galton--Watson tree.
This is the incipient infinite cluster for critical, finite variance
Galton--Watson trees \cite{kesten86sub}. Theorem~\ref{teouniforest}
then appears as a strengthening of this picture, valid only for Poisson
Galton--Watson trees, in which $k$ is allowed to grow with $n$.
\end{re*}

Recall that $T$ is a uniform Cayley tree on $[n]$ and that $\mathcal{X}
=(X_m)_{m \in\mathbb N}$ is a sequence of i.i.d. uniform elements of $[n]$.
In the next proposition, which is \mbox{essentially} a time-reversed version
of Theorem~\ref{teokey}, we write $\mathbf F(T,\mathcal{X}) = \mathbf
F$ for readability.
%
\begin{prop}\label{propreverse}
For any forest $\mathbf{f}=(t_1,\ldots,t_k)$ on $[n]$, given that
$\mathbf F
=\mathbf{f}$, independently for each $i \in[k-1]$ the parent
$a(r(t_i),T)$ of $r(t_i)$ in $T$ is a uniformly random element of
$\bigcup_{j=i+1}^k v(t_j)$.
\end{prop}
\begin{pf}
If $k=1$, then there is nothing to prove. If $k > 1$, then
fix any sequence $\mathbf{v}=(v_1,\ldots,v_{k-1})$ with $v_i \in
\bigcup_{j=i+1}^k v(t_j)$ for each $i \in[k-1]$. Write $t(\mathbf
{f},\mathbf
{v})$ for the tree formed from $\mathbf{f}$ by adding an edge from
$r(t_i)$ to $v_{i}$ for each $i \in[k-1]$.
In order that $\mathbf F=\mathbf{f}$ and that, for each $i \in[k-1]$,
$a(r(t_i),T)=v_{i}$, it is necessary and sufficient that $T=t(\mathbf
{f},\mathbf{v})$ and that for each $i \in[k]$, $X_{\sigma_i} = r(t_i)$.
The probability that $T=t(\mathbf{f},\mathbf{v})$ is $n^{-(n-1)}$.
Furthermore, since $(X_m)_{m \in\mathbb N}$ are i.i.d. elements of $[n]$,
\[
{\mathbf P}\bigl(X_{\sigma_i}=r(t_i), 1\le i\le k | T=t(
\mathbf{f}, \mathbf{v})\bigr) = \prod_{i \in[k]}
\frac{1}{|\bigcup_{j \geq i} v(t_j)|}.
\]
It follows that
\[
{\mathbf P}\bigl(\mathbf F=\mathbf{f}\mbox{ and } a
\bigl(r(t_i),T \bigr)=v_{i}, 1\le i<k\bigr) =
\frac{1}{n^{n-1}} \cdot\prod_{i \in[k]} \frac{1}{|\bigcup_{j \geq i} v(t_j)|},
\]
which proves the proposition since this expression does not depend on
$v_1, \ldots,\break v_{k-1}$.
\end{pf}

Theorem~\ref{teoreverse} is an immediate consequence of
Proposition~\ref{propreverse}.

\section{Isolating more than one vertex}\label{secmorethanone}
In this section we describe how to generalize the arguments of
Section~\ref{sectrefor} to
obtain results on isolating sets of vertices of size greater than one.
Recall that when performing the planted cutting of $S$ in $t$,
described in Section~\ref{secintro}, we
wrote $W=\{w_1,\ldots,w_k\}$ for the set of new vertices, and wrote
$M=M(t,S)$ for the (random) total number of edges removed.
In order to study the random variable $M$, it turns out to be necessary
to study a transformation of the planted cutting procedure.
The modified procedure is defined via a canonical re-ordering of the
sequence of removed edges.
As such, it may be coupled with the original procedure so that the
final set of removed edges is the same in both.
In particular, both procedures isolate the vertices of $W$, and the
total number of cuts has the same distribution in both.

In the following, for an edge $e$ and a connected component $C$, we
write $e\in C$ to mean that \emph{both} endpoints of $e$ lie in $C$, or
equivalently (since the connected components are trees) that the
removal of $e$ leaves $C$ disconnected.
Also, recall from Section~\ref{secnotation} that given a set $A$ of
edges, we write $t\setminus A$ for the graph \mbox{$(v(t),e(t)\setminus A)$}.

Now fix a sequence $\mathbf{e}=(e_1,\ldots,e_m)$ of distinct edges of
$t$. We
say that $\mathbf{e}$ is a \emph{possible cutting sequence} (for $S$
in $t$) if:
\begin{itemize}
\item each edge $\{v_i,w_i\}$, $1 \le i \le k$ appears in $\mathbf{e}$
($\mathbf{e}$
really isolates $w_1,\ldots, w_k$), and
\item for each $1 \le j \le m$, one has $e_j\in C(W, t\setminus\{
e_1,\ldots, e_{j-1}\})$, that is, each $e_j$ indeed produces a cut.
\end{itemize}
We now describe a canonical re-ordering of $\mathbf{e}$, which we
denote $\mathbf{e}
^*$; this re-ordering operation gives rise to the modified cutting
procedure. In $\mathbf{e}^*$, we first list all edges whose removal decreases
the size of the component containing $w_1$ (in increasing order of
arrival time). We then list all \emph{remaining} edges whose removal
decreases the size of the component containing $w_2$, again in
increasing order of arrival time, and so on. (This is somewhat related
to a size-biased reordering of an exchangeable random structure; see
\cite{Pitman2006}, Chapter~1. The next three paragraphs formalize this
description.)

For $1 \le i \le k$, write
\[
U_i = U_i(\mathbf{e}) = \bigl\{j\dvtx e_{j}
\in C \bigl(w_i, t\setminus\{e_1,\ldots,
e_{j-1}\} \bigr) \bigr\}
\]
and let $U_i^* = U_i \setminus(\bigcup_{j=1}^{i-1} U_j)$. In words,
$U_i^*$ is the set of times $j$ at which the component containing $w_i$
does not contain any of $w_1,\ldots,w_{i-1}$, and such that removing
the current edge $e_j$ decreases the size of this component.

Next, let $m(i) = m(i,t,\mathbf{e}) = |U_i|$, write $Z_i=Z_i(\mathbf
{e}) =
(z_{i,1},\ldots,z_{i,m(i)})$ for the sequence obtained by listing the
elements of $U_i$
in increasing order, and define $Z_i^*$ accordingly.
Notice that once $w_i$ is in a component distinct
from $w_1,\ldots,w_{i-1}$, it~can never rejoin such a component, and so
writing $s(i)=s(i,t,\mathbf{e})=\min\{\ell\dvtx z_{i,\ell} \in
U_i^*\}$, we must have
\[
Z_i^* = (z_{i,s(i)},z_{i,s(i)+1},\ldots,z_{i,m(i)}).
\]
We then write
\begin{eqnarray*}
\mathbf{e}^* & =& (e_{z_{1,s(1)}},\ldots,e_{z_{1,m(1)}},e_{z_{2,s(2)}},
\ldots,e_{z_{2,m(2)}},\ldots,e_{z_{k,s(k)}},\ldots,e_{z_{k,m(k)}})
\\
& =& \bigl(e_1^*,\ldots,e_m^* \bigr),
\end{eqnarray*}
the latter equality constituting the definition of $e_1^*,\ldots,e_m^*$.
For $1 \le i \le k$, let $a_i(t,\mathbf{e}^*)= 1+ \sum_{\ell
=1}^{i-1} (m(\ell
)-s(\ell)+1)$ let $b_i(t,\mathbf{e}^*) = \sum_{\ell=1}^{i} (m(\ell
)-s(\ell)+1)$,
and set
\[
\mathbf{e}^*_i = \bigl(e^*_j, a_i \le j \le
b_i \bigr) = \bigl(e_{z_{i,j}},s(i) \le j \le m(i) \bigr).
\]
We remark that necessarily $e_{z_{i,m(i)}} = \{w_i,v_i\}$, and so in
particular the sequence $\mathbf{e}^*_i$ is nonempty for each $1 \le i
\le k$.

Now write $\mathbf E=\mathbf E(t,S)=(E_1,\ldots,E_M)$ for the random
sequence of
removed edges (in the original planted cutting procedure),
write $\mathbf E^*=\break \mathbf E^*(t,S)=(E_1^*,\ldots,E_M^*)$ for the
rearrangement of
$\mathbf E$ described above, and likewise define $\mathbf E^*_i$, for
$1 \le i \le
k$, as above.

It is easily seen that if $\mathbf{e}$ is not a possible cutting
sequence, then ${\mathbf P}(\mathbf E(t,S)=\mathbf{e})=0$, and if
$\mathbf
{e}$ \emph{is} a possible
cutting sequence, then
%
\begin{eqnarray}
{\mathbf P}\bigl(\mathbf E(t,S)=\mathbf{e}\bigr) &=&\prod
_{j=1}^{m} \frac{1} {|e(C(W, t\setminus\{e_1,\ldots, e_{j-1}\}
))|}.
\end{eqnarray}
For our purposes, it is in fact the expression for ${\mathbf
P}(\mathbf E^*(t,S) = \mathbf{e} ^*)$ given in the following lemma
that will be more useful.
Fix any sequence $\mathbf{f}=(f_1,\ldots,f_m)$ of edges of $t\langle
S \rangle$. If
there exists a possible cutting sequence $\mathbf{e}=(e_1,\ldots,e_m)$ for
$S=(v_1,\ldots,v_k)$ in $t$ such that $\mathbf{e}^*=\mathbf{f}$,
then we say that
$\mathbf{f}$ is \emph{valid} (\textit{for $t$ and $S$}).
%
\begin{lem}
Given any sequence $\mathbf{f}=(f_1,\ldots,f_m)$ that is valid for
$t$ and
$S$, we have
\[
{\mathbf P}\bigl(\mathbf E^*(t,S) = \mathbf{f}\bigr) = \prod
_{i=1}^k \prod_{j=a_i(t,\mathbf{f})}^{b_i(t,\mathbf{f}
)}
\frac{1}{|e(C(w_i,t\setminus\{f_1,\ldots,f_{j-1}\}))|}.
\]
\end{lem}
\begin{pf}
We prove the lemma by induction on $|e(t\langle S \rangle)|$. Fix
$\mathbf
{f}$ as in
the statement of the lemma, write
\[
\mathcal{E}(\mathbf{f})=\mathcal{E}(\mathbf{f},t,S)= \bigl\{ \mathbf{e}\dvtx
\mathbf{e}\mbox{ is a possible cutting sequence for $S$ in $t$ and }
\mathbf{e}^*=\mathbf{f} \bigr\}
\]
and note that $\mathbf{f}\in\mathcal{E}(\mathbf{f})$. For any
$\mathbf{e}=(e_1,\ldots,e_m) \in
\mathcal{E}(\mathbf{f})$ we necessarily have \mbox{$e_1=f_1$}, and so
\[
{\mathbf P}\bigl(\mathbf E_1^*(t,S) =f_1\bigr) = {
\mathbf P}\bigl(\mathbf E_1(t,S) =f_1\bigr) =
\frac{1}{|e(t\langle S \rangle)|}.
\]
If $e_1=\{v_1,w_1\}$, then writing $S'=(v_2,\ldots,v_k)$, we have
\begin{eqnarray*}
{\mathbf P}\bigl(\mathbf E^* =\mathbf{f}| \mathbf E^*_1
=f_1\bigr) & =& {\mathbf P}\bigl(\mathbf E^* =\mathbf{f}| \mathbf
E_1 =f_1\bigr)
\\
& =& {\mathbf P}\bigl(\mathbf E^* \bigl(t,S' \bigr)=(f_2,
\ldots,f_m)\bigr)
\end{eqnarray*}
and the result follows by induction since $t\langle S' \rangle$ has fewer
edges than $t\langle S \rangle$.\vadjust{\goodbreak}

If $e_1 \ne\{v_1,w_1\}$, then
write $t_1=C(w_1,t\langle S \rangle\setminus\{e_1\})$, and write
$t_2$ for
the other component of $t\langle S \rangle\setminus\{e_1\}$;
each of these trees has fewer edges than $t\langle S \rangle$.
Write $S_1=(x_1,\ldots,x_{k_1})$ and $S_2=(y_1,\ldots,y_{k_2})$ for the
nodes of $S$ within $t_1$ and $t_2$, respectively, listed in the same
order as in $S$.

Now fix any possible cutting sequence $\mathbf{e}=(e_1,\ldots,e_m)$
with $e_1=f_1$.
Write $\mathbf{e}^{(1)}$ and $\mathbf{e}^{(2)}$ for those edges in
the sequence
$(e_2,\ldots,e_m)$ falling in $t_1$ and $t_2$, respectively, and listed
in the same order as in $\mathbf{e}$.
Then it is clear that, conditionally on $\mathbf E_1=f_1$, the
sequences $\mathbf E
(t_1,S_1)$ and $\mathbf E(t_2,S_2)$ have the distribution of the planted
cutting procedure on $t_1\langle S_1 \rangle$ and $t_2\langle S_2
\rangle$,
respectively, and are independent. In other words,
\begin{eqnarray*}
&&\mathbf{P}\bigl( \mathbf E(t_1,S_1)=
\mathbf{e}^{(1)}, \mathbf E(t_2,S_2)=
\mathbf{e}^{(2)} | \mathbf E_1=f_1\bigr)
\\
&&\qquad =  {
\mathbf P}\bigl(\mathbf E(t_1,S_1)=\mathbf{e}^{(1)}
\bigr) \cdot{\mathbf P}\bigl(\mathbf E(t_2,S_2)=
\mathbf{e}^{(2)}\bigr).
\end{eqnarray*}
Furthermore, if $\mathbf{e}\in\mathcal{E}(\mathbf{f})$, then
$e_1=f_1$, and $\mathbf{e}\in\mathcal{E}
(\mathbf{f})$
if and only if $\mathbf{e}^{(1)}\in\mathcal{E}(\mathbf
{f}^{(1)},t_1,S_1)$ and $\mathbf{e}^{(2)}
\in\mathcal{E}(\mathbf{f}^{(2)},t_2,S_2)$.
[Note: this does not mean that the map from $\mathbf{e}$ to $(\mathbf
{e}^{(1)},\mathbf{e}
^{(2)})$ is bijective! In fact, for a given pair $\mathbf{e}^{(1)}\in
\mathcal{E}(\mathbf{f}
^{(1)},t_1,\break S_1)$ and $\mathbf{e}^{(2)} \in\mathcal{E}(\mathbf
{f}^{(2)},t_2,S_2)$, the
number of pre-images in $\mathcal{E}(\mathbf{f})$ is precisely ${m -
1 \choose m_1}$,
where $m_1$ is the length of $\mathbf{f}^{(1)}$.] Also, $\mathbf
{f}^{(1)}$ (resp.,
$\mathbf{f}^{(2)}$) is valid for $t_1$ and $S_1$ (resp., for $t_2$ and $S_2$).
It follows that
\begin{eqnarray*}
&& {\mathbf P}\bigl(\mathbf E^*=\mathbf{f}|\mathbf E_1=f_1
\bigr)
\\
&&\quad= \sum_{\mathbf{e}\in\mathcal{E}} {\mathbf P}(\mathbf E=
\mathbf{e}| \mathbf E_1=f_1)
\\
&&\quad= \sum_{\mathbf{e}^{(1)}\in\mathcal{E}(\mathbf
{f}^{(1)},t_1,S_1)} \sum
_{\mathbf{e}
^{(2)}\in\mathcal{E}(\mathbf{f}^{(2)},t_2,S_2)} {\mathbf P}\bigl(\mathbf E(t_1,S_1)=
\mathbf{e}^{(1)},\mathbf E(t_2,S_2)=
\mathbf{e}^{(2)}| \mathbf E_1=f_1\bigr)
\\
&&\quad= \sum_{\mathbf{e}^{(1)}\in\mathcal{E}(\mathbf
{f}^{(1)},t_1,S_1)} \sum
_{\mathbf{e}
^{(2)}\in\mathcal{E}(\mathbf{f}^{(2)},t_2,S_2)} {\mathbf P}\bigl(\mathbf E(t_1,S_1)=
\mathbf{e}^{(1)}\bigr) \cdot{\mathbf P}\bigl(\mathbf
E(t_2,S_2)= \mathbf{e}^{(2)}\bigr)
\\
&&\quad= {\mathbf P}\bigl(\mathbf E^*(t_1,S_1)=
\mathbf{f}^{(1)}\bigr) \cdot{\mathbf P}\bigl(\mathbf E^*(t_2,S_2)
= \mathbf{f}^{(2)}\bigr)
\end{eqnarray*}
from which the result again follows by induction.
\end{pf}

The formula in the preceding lemma implies that
removing edges in the order given by $\mathbf E^*$ corresponds to the
following procedure. For each $1 \le i \le k$, in that order, remove
edges of $t$ uniformly at random from among those whose removal reduces
the size of the component currently containing $w_i$, until $w_i$ is isolated.
We call this the \emph{ordered cutting} of $S$ in $t$.

For $1 \le i \le k$, write $M_i$ for the random time at which $w_i$ is
isolated in the ordered cutting procedure
\begin{eqnarray*}
M_i&=&M_i(t,S) =\max \bigl\{j\dvtx E_j^*
\in C \bigl(w_i,t\setminus \bigl\{E_1^*,
\ldots,E_{j-1}^* \bigr\} \bigr) \bigr\}
\\
& =& \min \bigl\{j\dvtx\bigl|C \bigl(w_i,t\setminus \bigl
\{E_1^*, \ldots,E_{j}^* \bigr\} \bigr)\bigr| =0 \bigr\}
\end{eqnarray*}
(recall that the counting does not include planted vertices),
and note that $M_1<M_2<\cdots<M_k\stackrel{d}{=}M$.

Now, let $T$ be a uniform Cayley tree on $[n]$, let $V_1,\ldots,V_k$ be
independent, uniformly random elements of $[n]$, and let
$S_k=(V_1,\ldots,V_k)$.
Then write $M_k=M(T,S_k)$ for the number of edges removed during the
ordered cutting of $S_k$ in $t$.
%
\begin{teo}\label{teoknodes}
$M_k-k$ is distributed as the number of edges spanned by the root plus
$k$ independent, uniformly random nodes in a uniform Cayley tree of
size~$n$.
\end{teo}
Theorem~\ref{teokcoup} follows immediately from Theorem~\ref
{teoknodes} and the relationship between planted cutting and ordered
cutting described above.
To prove Theorem~\ref{teoknodes}, we will exhibit a coupling which
generalizes that of Section~\ref{sectrefor} and which we now explain.
The coupling hinges upon the following, easy lemma, whose proof is
omitted. Recall that if $S$ is a set of nodes in a tree $t$, then
$t{\llbracket S \rrbracket}$ is the subtree of $S$ spanned by $S$.
%
\begin{lem}\label{lemrandsubtree}
Fix $i \ge1$. Let $T$ be a uniform Cayley tree on $[n]$, let
$V_1,\ldots,V_{i+1}$ be independent, uniformly random elements of
$[n]$, and let $S=\{r(T),V_1,\ldots,V_i\}$.
Let $U$ be the most recent ancestor of $V_{i+1}$ in $T$ which is an
element of $v(T{\llbracket S \rrbracket})$.
Let $R$ be the set of nodes whose path to $V_{i+1}$ uses no edges of
$T{\llbracket S \rrbracket}$ (such paths may pass through $U$).
Let $T^+=T{\llbracket R \rrbracket}$, let $T^-=T{\llbracket
([n]\setminus R)\cup\{U\} \rrbracket}$ and root
$T^+$ and $T^-$ at $U$ and at $r(T)$, respectively.
Then conditionally on $R$, $T^+$ is a uniformly random labeled rooted
tree on $R$, independent of $T^-$ and of $V_1,\ldots,V_i$, and
$V_{i+1}$ is a uniformly random element of $R$ independent of $T^+,T^-$
and $V_1,\ldots,V_i$.
\end{lem}
The definitions in Lemma~\ref{lemrandsubtree} are depicted in
Figure~\ref{figlemma8}.
%
\begin{figure}

\includegraphics{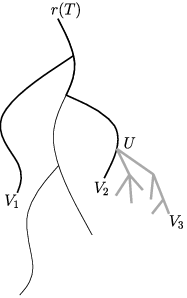}

\caption{An example of the definitions of Lemma~\protect\ref
{lemrandsubtree} in the case $i=2$ [so $S=(r(T),V_1,V_2)$]. The
subtree $T{\llbracket S \rrbracket}$ is in thicker black lines. The
tree $T^+$ is in
thick grey lines, and the tree $T^-$ consists of all black lines (thick
and thin).}
\label{figlemma8}
\end{figure}

\begin{pf*}{Proof of Theorem \ref{teoknodes}}
We provide a coupling between the random sequence of edges $\mathbf E
^*(T,(V_1,\ldots,V_k))$ and a sequence $T_1,\ldots,T_k$ of trees on
$[n]$, such that the following properties hold.
First, for
any rooted tree $t$ on $[n]$, and any $v_1,\ldots,v_i$ elements of
$[n]$ (not necessarily distinct),
%
\begin{equation}
\label{equniformonk} {\mathbf P}(T_i=t, V_{1}=v_1,
\ldots,V_{i}=v_i) = n^{-(n-1+i)}.
\end{equation}
Second, for each $1 \le i \le k$, the following holds:
\begin{itemize}[($\star$)]
\item[($\star$)] the forest obtained from $T\langle(V_1,\ldots,V_i)
\rangle$
by first removing all edges of $\{E^*_1,\ldots,E^*_{M_i}\}$, then
deleting $w_1,\ldots,w_i$, is identical
to the forest obtained from $T_i$ by removing all edges of its subtree
$T_i{\llbracket r(T_i),V_1,\ldots,V_i \rrbracket}$.
\end{itemize}
Equation (\ref{equniformonk}) says that $T_i$ is a uniform Cayley tree
and $V_1,\ldots,V_i$ are independent of $T_i$, and ($\star$) then
implies in particular (by considering only the case $i=k$) that $M_k-k$
is equal to the number of edges of $T_k{\llbracket r(T_k),V_1,\ldots,V_k \rrbracket}$.
This clearly implies the theorem, and so it remains to explain how we
construct such a sequence.

Fix a sequence $\mathcal{X}=(X_i)_{i \ge1}$ of i.i.d. uniform
elements of $[n]$.
Let $T_1$ be the tree built by running the modified Aldous--Broder
dynamics on $T^{r \leftrightarrow V_1}$ (recall that this is the tree
$T$, rerooted
at node $V_1$) with the sequence $(X_i)_{i \ge1}$. [In the notation of
Section~\ref{sectrefor},
$T_1=\widehat{T}(T^{r \leftrightarrow V_1},\mathcal{X})$.] By
Theorem~\ref{teokey}, for any
tree $t$ on $[n]$ and any $v \in[n]$, ${\mathbf
P}(T_1=t,V_1=v)=n^{-n}$, so
(\ref{equniformonk}) holds in the case $i=1$.
Temporarily write $u_1,\ldots,u_\ell$ for the nodes on the path in
$T_1$ from $r(T_1)$ to $V_1$, in the same order they appear on that
path. We must then have $u_\ell=V_1$,
and $M_1=\ell$. For $1 \le j \le\ell-1$, let $E_j^*=\{u_j,a(u_j,T^{r
\leftrightarrow V_1})\}$, and note that this is also an edge of $T$
since $T$ and
$T^{r \leftrightarrow V_1}$ have the same edge set.
Then let $E_{M_1}^* = \{u_{\ell},w_1\}=\{V_1,w_1\}$. (An example of
this construction is shown in Figure~\ref{figmultdyn}.)
By construction, it is immediate that ($\star$) then holds in the case $i=1$.

%
\begin{figure}

\includegraphics{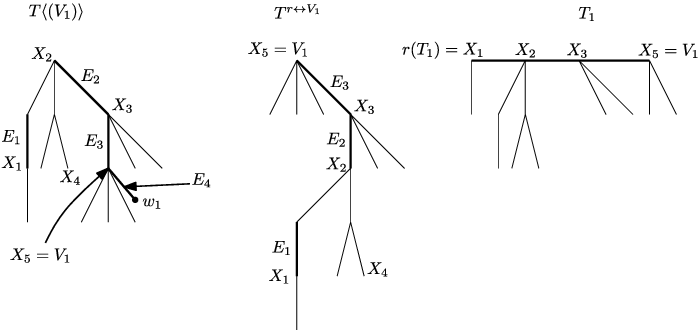}

\caption{Left: the tree $T\langle(V_1)\rangle$. Center: the tree $T^{r
\leftrightarrow V_1}$, planted at $V_1$.
Right: the tree $T_1$.
The vertex and edge labels provide an example of the construction in
the proof of Theorem~\protect\ref{teoknodes}, in the case
$k=1$. For each of the three trees, the forest obtained by removing the
bold edges [and, for $T\langle(V_1)\rangle$, then throwing away the
vertex $w_1$] is identical.}\label{figmultdyn}
\end{figure}

Now fix $1 \le j < k$, suppose that $T_1,\ldots,T_j$ and
$(E_1^*,\ldots,E_{M_j}^*)$ are already defined and that (\ref
{equniformonk}) and
$(\star)$ both hold for each $1 \le i \le j$.
As defined, $V_{j+1}$ is independent of $T_i$ and of $(E_1^*,\ldots,E_{M_j}^*)$, and so for any tree $t$ on $[n]$ and any sequence
$u_1,\ldots,u_{j+1}$ of elements of $[n]$, we have
%
\begin{equation}
\label{eqinductivestep} {\mathbf P}(T_j=t, V_{1}=u_1,
\ldots,V_{j+1}=u_{j+1}) = n^{-(n-1+i+1)}.
\end{equation}
Let $U$ be the most recent ancestor of $V_{j+1}$ that lies in
$T_j{\llbracket r(T_i),V_1,\ldots,V_j \rrbracket}$, and define $T^+$
and $T^-$
as in Lemma~\ref{lemrandsubtree}.

Now let $\mathcal{X}'$ be a random sequence such that conditionally on
$v(T^+)$, the entries of $\mathcal{X}'$ are independent uniform
elements of
$v(T^+)$, independent of
all preceding randomness. Then apply the modified Aldous--Broder
dynamics to $T^{+,r \leftrightarrow V_{j+1}}$, and call the result $T^*$.
By Theorem~\ref{teokey}, given that $v(T^+)=S$, $(T^+,V_{j+1})$ and
$(T^*,V_{j+1})$ are identically distributed.
As above, let $u_1,\ldots,u_\ell$ be the nodes on the path from
$r(T^*)$ to $V_{j+1}$, and note that we must have $M_{j+1}=M_j+\ell$.
For $1 \le i \le\ell-1$ let
$E^*_{M_i+j} = \{u_i,a(u_i,T^{+,r \leftrightarrow V_{j+1}})\}$, and let
$E^*_{M_{j+1}}=\{V_{j+1},w_{j+1}\}$.
In words, we have applied exactly the same construction as in the case
$i=1$, but to the subtree $T^+$ of $T$ (which contains $V_{j+1}$).
Figures~\ref{figlemma8} and~\ref{figmultdyn} may be useful as visual
aids to these definitions.

Write $P$ for the parent of $U$ in $T_j$, and $C_1,\ldots,C_{\ell}$ for
the children of $U$ in $T_j\setminus T^+$ (any such child is an
ancestor of at least one of $V_1,\ldots,V_j$). Now let $T_{j+1}$ be the
tree obtained from $T_j$ by replacing $T^+$ by $T^*$.
In other words, $T_{j+1}$ is built from $T_j$ by, first, removing all
edges of $T_j$ that are incident to nodes of $T^+$ and then, second,
adding all edges of $T^*$ as well as edges from the root of $T^*$ to
$P$ and to each of $C_1,\ldots,C_{\ell}$.
With this construction, $(\star)$ now holds for all $1 \le i \le j+1$.

Finally, write $R=v(T^+)$. By Lemma~\ref{lemrandsubtree} and by
Theorem~\ref{teokey}, $(T^+,V_{j+1})$ and $(T^*,V_{j+1})$ are
identically distributed conditionally on their vertex sets, and both
are independent of $T^-$ and of $V_1,\ldots,V_j$. It follows that
(\ref
{eqinductivestep}) still holds with $T_j$ replaced by $T_{j+1}$, and
this verifies (\ref{equniformonk}) and completes the proof by induction.
\end{pf*}

\section{A novel transformation of the Brownian CRT}\label{secnovel}

In \cite{Janson2006c}, Janson suggested that it should be possible to
define a version of the cutting procedure directly on $\mathcal{T}$.
In this
section, we provide such a
construction. This construction yields straightforward, ``conceptual''
proofs of some of the main results of \cite{Janson2006c}, and also
provides a novel, reversible transformation
from $\mathcal{T}$ to another, doubly-rooted Brownian CRT.
(We remark in passing that the results of this section can also be
straightforwardly used to prove the first convergence result from
Theorem 1.10 of \cite{Janson2006c}.) Using the by now well-known coding
of the Brownian CRT by a standard Brownian excursion, this
transformation can be viewed as a new, invertible random transformation
between Brownian excursion and Brownian Bridge.

We now describe the details of the construction, using the language of
\mbox{$\mathbb{R}$-}trees.
For the interested reader, we describe the corresponding transformation
from Brownian excursion to reflecting Brownian bridge in Appendix~\ref{secapp}.

We begin with a quick, high-level description of the transformation. An
initial compact real tree $\mathcal{T}$ distributed as the Brownian
CRT will be
cut by points falling on its skeleton. When a point arrives, the
current tree is separated into two connected components; the one
containing the root will suffer further cuts at later times, while the
other one---the pruned tree---will no longer be cut. As in the discrete
transformation of Section~\ref{sectrefor}, the cut trees are
rearranged by attaching their roots to a ``backbone'' so as to form a
new real tree. We now describe the continuous transformation by first
building the backbone that will eventually connect the roots of the
pruned subtrees, and then specifying where these subtrees should be
grafted along the backbone.

\subsection{The details of the transformation}\label{sectransdetails}

Let $\mathcal{P}$ be a Poisson process on $\operatorname
{skel}(\mathcal{T}) \times
[0,\infty)$ with
intensity measure $\ell\otimes\ddd t$, and for each $t \geq
0$, let
\[
\mathcal{P}^{\circ}_t = \bigl\{x \in\mathcal{T}\dvtx\exists
s, 0 \le s \le t, (x,s) \in\mathcal{P} \bigr\}.
\]
In \cite{aldous1998standard}, Aldous and Pitman used the point process
$\mathcal{P}$ to construct (what is now called) a self-similar fragmentation
process on $\mathcal{T}$ \cite{Bertoin2006}. For each $t \geq0$, let
$\mathcal{F}
^{\circ}_t = \mathcal{T}\setminus\mathcal{P}^{\circ}_t$.
In particular, two points $u,v \in\mathcal{T}\setminus\mathcal
{P}^{\circ}_t$ are in
the same component of $\mathcal{F}^{\circ}_t$ precisely if, in
$\mathcal{T}$, the path
${\llbracket u,v \rrbracket}$ contains no element of $\mathcal
{P}^{\circ}_t$. Aldous and Pitman
\cite{aldous1998standard} established many beautiful facts
about how the collection of masses of the components of $\mathcal
{F}^{\circ}_t$
evolve with $t$; one basic fact from \cite{aldous1998standard}
is that a.s., for each $t>0$, $\mathcal{F}^{\circ}_t$ has only
countably many
components, and the total mass of all components of $\mathcal
{F}^{\circ}_t$ is
one. (This seems intuitively obvious, but note that it is a priori
possible that for every $t>0$, $\mathcal{F}^{\circ}_t$ contains uncountably
many components, each of mass zero; consider $[0,1]\setminus\mathbb Q$.)

\begin{db*}
For $t \ge0$, write
$\widetilde{\mathcal{T}}_{t}$ for the component of~$\mathcal{F}^{\circ
}_t$ containing the root
$\rho$ at time $t$; then define a process $(L(t),t \ge0)$ by setting
%
\begin{equation}
\label{eqdeflocaltime} L(t) = \int_0^t \mu(\widetilde{\mathcal{T}}_s) \,\ddd s.
\end{equation}
The process $L(t)$ is the continuum analogue of the ``number of cuts by
time $t$''; the process $(L(t), t \ge0)$ will code the distance along
the backbone in the continuum transformation; see Theorem~\ref
{teolocaltime} and Corollary~\ref{corlocaltime} below.

Theorem 6 of \cite{aldous1998standard} states that if we define an
increasing function $(X(t),t \geq0)$ by
%
\begin{equation}
\bigl( \mu(\widetilde{\mathcal{T}}_t), t \geq0 \bigr) =
\frac{1}{1+X(t)},
\end{equation}
then $X(\cdot)$ is a stable subordinator of index $1/2$, or in other words,
is distributed as the inverse local time process at zero of a standard
reflecting Brownian motion.
The function $X(\cdot)$ has almost sure quadratic growth, and it
follows that $L(\infty):= \lim_{t \to\infty} L(t)$ is almost
surely finite.
[The proof of Corollary~\ref{corlocaltime}, below, contains a
different proof that $L(\infty)$ is almost surely finite, using the
principle of accompanying laws.]
\end{db*}

\begin{aaa*}
Since $\mathcal{P}$ is a countable set, we may enumerate
its atoms
as $((p_i,\tau_i),i \in\mathbb N)$.
For $t \ge0$, let
\[
I_t = \bigl\{i \in\mathbb N\dvtx0 \le\tau_i \le t, \mu(
\widetilde{ \mathcal{T}}_{\tau_i}) < \mu(\widetilde{\mathcal{T}}_{\tau_i-})
\bigr\}
\]
and let
\[
\mathcal{P}_t = \{p_i\dvtx i \in I_t\}
\subseteq \mathcal{P}^{\circ}_t.
\]
Let $P_{\infty} = \bigcup_{t \ge0} \mathcal{P}_t$,
and let $I_{\infty}=\bigcup_{t\ge0} I_t$.
Next, for $0 \le t \le\infty$, let $\widetilde{\mathcal{F}}_t= \mathcal
{T}\setminus\mathcal{P}_t$,
let $\tilde{d}_t$ be its intrinsic distance: for points $x,y$ in the
same component of $\widetilde{\mathcal{F}}_t$, we have \mbox{$\tilde{d}_t(x,y)=d(x,y)$},
while for $x,y$ in distinct components of $\widetilde{\mathcal{F}}_t$, we have
\mbox{$\tilde{d}_t(x,y)=\infty$},\footnote{See \cite{burago01}, Sections~2.3
and 2.4, for the general definition of intrinsic distance for a subset
of a metric space.} and let $\tilde{\mu}_t$ be the restriction of $\mu$ to $\widetilde
{\mathcal{F}}_t$.
Then let $(\mathcal{F}_t,d_t)$ be the metric space completion of
$(\widetilde{\mathcal{F}
}_t,\tilde{d}_t)$, and let $\mu_t$ be\vspace*{1pt} the extension of $\tilde{\mu}_t$ obtained
by assigning measure zero to all points of
$\mathcal{F}_t\setminus\widetilde{\mathcal{F}}_t$; note that there are
only countably many
such points.\footnote{The assiduous reader may ask: the forest
$(\mathcal{F}
_t,d_t,\mu_t)$ is meant to be a random element of what (Polish) space?
One possible answer is to view this forest as given by some random
function $e_t\dvtx[0,1] \to[0,\infty)$ with $e_t(0)=e_t(1)=0$,
and with the ``components'' of the forest separated by the zeros of
$e_t$; this perspective is elaborated in Appendix~\ref{secapp}.
However, this forest itself is essentially introduced for expository
purposes and plays no role in the sequel; as such, the details
of how to formalize the definition of $(\mathcal{F}_t,d_t,\mu_t)$ are
unimportant in the remainder of the paper.}

Next, write $\mathcal{T}_t$ for the component of $\mathcal{F}_t$
containing $\rho$.
We then have that a.s. for all $t \ge0$, $\widetilde{\mathcal{T}}_t$ is a
connected component of $\widetilde{\mathcal{F}}_t$, and that 
a.s.%
\begin{equation}
\label{eqsamemasses} \bigl(\mu(\widetilde{\mathcal{T}}_t),t \ge0 \bigr) =
\bigl( \mu_t(\mathcal{T}_t), t \ge0 \bigr).
\end{equation}
By definition, a.s. for every $0 \le s < t$, every component of $\widetilde
{\mathcal{F}}_s$ not containing $\rho$ is also a component of $\widetilde
{\mathcal{F}}_t$.
This naturally extends to the completions $\mathcal{F}_s$ and
$\mathcal{F}_t$.\vspace*{1pt}

For $0 \le t \le\infty$, let $\tilde{\phi}_t$ be the identity map from
$\widetilde{\mathcal{F}}_t$ to $\mathcal{T}$,
and let $\phi_t$ be the unique extension of $\tilde{\phi}_t$ to
$\mathcal{F}_t$
whose restriction to any
component of $\mathcal{F}_t$ is a continuous function.
With probability one, for each $i \in I_\infty$, $p_i$ has degree two
in $\mathcal{T}$ and also in
$\mathcal{F}_{\tau_i-}$.
It follows that almost surely, for each $i \in I_{\infty}$, $\mathcal
{F}_{\tau
_i}\setminus\mathcal{F}_{\tau_i-}$ contains precisely two points.
Call these points $x_i$ and $y_i$, labeled so that $x_i \notin\mathcal
{T}_{\tau
_i}$ and $y_i \in\mathcal{T}_{\tau_i}$. Write $f_i$ for the
component of $\mathcal{F}
_{\tau_i}$ containing $x_i$.
\label{eqeffi}
Necessarily, $x_i \in f_i\setminus\widetilde{\mathcal{F}}_t$ and
$p_i=\phi_t(x_i)$
is the closest point of
$\phi_t(f_i)$ to $\rho$; in other words, $p_i$ is ``the root of the
subtree cut at time $\tau_i$.'' Also, $x_i$ and $y_i$ are both leaves
in $\mathcal{F}_{\tau_i}$.
For distinct points $p_i,p_j \in I_t$ the trees $f_i,f_j$ are disjoint,
so in particular $x_i \neq x_j$.

The space $(\mathcal{F}_{\infty},d_{\infty},\mu_{\infty})$ is the limiting
analogue of the forest $\mathbf F$ from Section~\ref{sectrefor}.
We note that $(\mathcal{T},\mu)$ can be recovered from $(\mathcal
{F}_{\infty},d_{\infty
},\mu_{\infty})$ by identifying $x_i$ and $y_i$ for each $p_i \in
I_{\infty}$, and taking as measure the corresponding push-forward of
$\mu_{\infty}$.

For $0 \le t \le\infty$, let $A_t$ be the real tree consisting of the
line segment $[0,L(t)]$ with the standard distance. Then form a
measured $\mathbb{R}$-tree $(\widehat{\mathcal{T}}_t,\hat{d}_t,\hat{\mu}_t)$
from $A_t$ and
$\mathcal{F}_t\setminus\mathcal{T}_t$,
by identifying $x_i \in f_i$ and $L(\tau_i) \in[0,L(t)]$, for each $i
\in I_t$, with measure $\hat{\mu}_{t}$ given by the push-forward of
$\mu
_t|_{\mathcal{F}_t\setminus\mathcal{T}_t}$. [We justify that $(\widehat{\mathcal{T}}_t,\hat{d}_t,\hat{\mu}_t)$ is indeed a well-defined random $\mathbb
{R}$-tree, using a
coding by excursions, in Appendix~\ref{secapp}.]
We naturally view these spaces as increasing in $t$. Write\vspace*{1pt} $\widehat
{\mathcal{T}
}=\widehat{\mathcal{T}}_{\infty}$, $\hat{d}=\hat{d}_{\infty}$, $\hat{\mu}=\hat{\mu
}_{\infty}$ and let $u=L(0)$ and $v=L(\infty)$. Almost surely both $u$
and $v$ are elements of $\widehat{\mathcal{T}}$.

The set of points of ${\llbracket u,v \rrbracket}$ of degree greater
than two in $(\widehat
{\mathcal{T}},\hat{d})$ are precisely the images in $\widehat{\mathcal{T}}$
of the points $\{x_i,i \in I_{\infty}\}$ in $\mathcal{F}_{\infty}$,
and if $x$ is the image of such a point $x_i$, then $\hat{d}(u,x)=L(\tau
_i)$. It follows that the set of times $\{\tau_i,i \in I_{\infty}\}$ is
measurable with respect to $(\widehat{\mathcal{T}},\hat{d}, \hat\mu)$.
Also, a.s.
$\{y_i, i \in I_{\infty}\}\cap\{x_i,i \in I_{\infty}\}=\varnothing
$, so
none of the points $\{y_i,i \in I_{\infty}\}$ are identified with other
points when forming $\widehat{\mathcal{T}}$. In other words, we may view
the points
$\{y_i,i \in I_{\infty}\}$ as points of $\widehat{\mathcal{T}}$ (rather
than as
\emph{members} of equivalence classes of points).

Now recall the definition of $(\mathcal{T},d,\mu,\rho)$ from the
start of the
section, and let $\rho'$ be a point of $\mathcal{T}$ selected
according to $\mu
$ and independent of $\rho$.
\end{aaa*}

%
\begin{teo}\label{teoexcbr}
It holds that $(\widehat{\mathcal{T}},\hat{d},\hat{\mu},(u,v))$ has
the same
distribution as $(\mathcal{T},d,\mu,(\rho,\rho'))$. Furthermore,
conditionally
on $(\widehat{\mathcal{T}},\hat{d},\hat{\mu},(u,v))$, the elements of
$\{y_i, i \in
I_{\infty}\}$ are mutually independent, and
for all $i \in I_{\infty}$, $y_i$ is distributed according to the
probability measure $\hat{\mu}|_{\widehat{\mathcal{T}}\setminus\widehat
{\mathcal{T}}_{\tau
_i}}/(1- \hat{\mu}(\widehat{\mathcal{T}}_{\tau_i}))$.
\end{teo}
We remark that Theorem~\ref{teoforwardtrans} is an immediate
consequence of the
first assertion of the theorem. Likewise,
Theorem~\ref{teoctsreverse} immediately follows from the definitions of
$(\widehat{\mathcal{T}},\hat{d},\hat{\mu},(u,v))$ and of the points $\{
y_i, i \in
I_{\infty}\}$ and from the second assertion of the theorem.

The remainder of Section~\ref{secnovel} is devoted to the proof of
Theorem~\ref{teoexcbr}.
The proof of Theorem~\ref{teoexcbr} relies on couplings with the
construction for uniform Cayley trees, and we introduce these couplings
in Section~\ref{seccouplings}. In Section~\ref{secconv-backbone}, we
show that the process $(L(t), t \ge0)$ is indeed the correct analogue
of ``number of cuts'' in the discrete setting. Finally, we wrap up the
proof of Theorem~\ref{teoexcbr} in Section~\ref{secproof-transform}.

\subsection{Some couplings between discrete and continuous
trees}\label{seccouplings}

The couplings we introduce in this section are not specific to the case
of uniform Cayley trees. This will be important in Section~\ref
{secsigma}, when we extend our results to other finite-variance
critical conditioned Galton--Watson trees.

Let $\xi=(\xi_i, i\ge0)$ be a critical finite-variance offspring
distribution, that is, a probability distribution on $\{0,1,\ldots\}$ with
\[
\sum_{i\ge0} i \xi_i =1 \quad\mbox{and}
\quad\sigma^2=\sum_{i\ge0} i(i-1)
\xi_i \in(0,\infty).
\]
In the following, we consider only values of $n$ such that a sum of $n$
i.i.d. random variables with distribution $\xi$ equals $n-1$ with
positive probability. For such $n\ge1$, let $T^n$ be a Galton--Watson
tree with offspring distribution $\xi$, conditioned to have\vadjust{\goodbreak} $n$ nodes.
For $x,y \in T^n$ let $d^{n}(x,y)$ be $\sigma n^{-1/2}$ times the graph
distance between $x$ and $y$ in $T^n$. Let $\rho^n$ denote the root of
$T^n$, let $\mu^n$ be the measure placing mass $1/n$ on each node of
$T^n$ and let $\ell^n$ be the measure placing mass $\sigma n^{-1/2}$ on
each vertex of $T^n$ (the ``discrete, rescaled length measure''). Let
next, $\mathcal{T}$ be the Brownian CRT with root $\rho$ and distance metric
$d$, let $\mu$ be its mass measure and let $\ell$ be the length measure
on the skeleton of $\mathcal{T}$. We will use the following
fundamental result heavily.

%
\begin{teo}[(Aldous \cite{Aldous1993a}, Le~Gall \cite
{legall06realtrees})]\label{teocayleygh}
It holds that
\[
\bigl(T^n, d^n,\mu^n,\rho^n
\bigr) \stackrel{d} {\to}(\mathcal{T},d,\mu,\rho)
\]
as $n \to\infty$, where convergence is in the 1-pointed
Gromov--Hausdorff--Prok\-horov sense.
\end{teo}
Strictly speaking, neither of the above papers establishes
Gromov--Haus\-dorff--Prokhorov convergence. However, deducing
Theorem~\ref
{teocayleygh} from the earlier results is essentially immediate; we
briefly sketch the line of the proof. First, by Proposition~10 of \cite
{miermont2009tessellations}, to prove Theorem~\ref{teocayleygh} it
suffices to establish convergence of $(T^n, d^n,\mu^n)$ to $(\mathcal
{T},d,\mu
)$ in the Gromov--Hausdorff--Prokhorov sense. Second, it is
straightforward to verify that Gromov--Hausdorff--Prok\-horov convergence
is equivalent to Gromov--Hausdorff convergence plus convergence of all
finite-dimensional marginals. The former convergence is established in
\cite{legall06realtrees}, and the latter is established in \cite{Aldous1993a}.
(See also Theorem~8 of Haas and Miermont \cite{HaMi2012a}, who
explicitly state Gromov--Hausdorff--Prokhorov
convergence as an application of their results on Markov branching trees.)

First, by Skorohod's representation theorem (see, e.g., \cite
{billingsley1968cpm}), we may consider a probability space $(\Omega,\mathcal{F},\mathbf{P})$ in which we have the almost sure GHP convergence
\[
\bigl(T^n,d^n,\mu^n,\rho^n
\bigr) \to(\mathcal{T},d,\mu,\rho).
\]
In such a space, we may find a sequence of correspondences $(R_n, n\ge
1)$ between $T^n$ and $\mathcal{T}$, such that $\operatorname
{dis}(R_n)\to0$ almost surely
as $n\to\infty$. We may also find a sequence of couplings $(\nu_n, n
\ge1)$ between $\mu^n$ and $\mu$ such that the defect \mbox{$D(\nu_n) \to0$}
almost surely as $n \to\infty$, and such that $\nu_n(R_n^c) \to0$
almost surely as $n \to\infty$.

Next, let $(s_i,i \geq1)$ be a random sequence of independent points
of $\mathcal{T}$ distributed according to $\mu$, and for each $n \in
\mathbb N$ let
$(s_{i}^n,i \geq1)$ be a sequence of independent points of $T^n$
distributed according to $\mu^n$. Also, write $s_0=\rho$ and
$s_{0}^n=\rho^n$ for notational convenience, and for $k \ge1$ write
$S^n_k = \{s_0^n,\ldots,s_k^n\}$. The almost sure GHP convergence
above implies
\cite{miermont2009tessellations}, Proposition 10, that for each fixed
$k \ge1$,
\[
\bigl(T^n, d^n,\mu^n, \bigl(s^n_0,
\ldots,s^n_k \bigr) \bigr) \stackrel{d} {\to} \bigl(
\mathcal{T},d, \mu,(s_0,\ldots,s_k) \bigr),
\]
in the sense of $d_{\mathrm{GHP}}^{k+1}$,
and Skorohod's theorem (applied once for each $k \ge1$) then implies
that we may work in a space in which almost surely, for all
$\varepsilon> 0$,
%
\begin{equation}
\label{eqskorohod2} \qquad\lim_{n \to\infty} \inf \bigl\{k\dvtx
d_{\mathrm{GHP}}^{k+1} \bigl( \bigl(T^n,\mu^n,
\bigl(s_{0}^n, \ldots,s_{k}^n
\bigr) \bigr), \bigl(\mathcal{T}, \mu,(s_{0},\ldots,s_{k})
\bigr) \bigr) \geq \varepsilon \bigr\} = \infty.
\end{equation}
For each $n,k \geq0$, recall that $T^n{\llbracket S^n_k \rrbracket}$
is the subtree of
$T^n$ spanned by $S^n_k$,
and let $\ell^n_k$ be the restriction of $\ell^n$ to $T^n{\llbracket
S^n_k \rrbracket}$.
Also, let $\mathcal{T}{\llbracket S_k \rrbracket}$ be the subtree of
$\mathcal{T}$ spanned by $S_k=\{
s_0,\ldots,s_k\}$, and let $\ell_k$ be the length measure on
$\mathcal{T}{\llbracket S_k \rrbracket}$.
In the space in which (\ref{eqskorohod2}) almost surely holds, we
immediately have
%
\begin{equation}
\label{eqskorohod3} \qquad\sup_{k \in\mathbb N} \lim_{n \to\infty}
d_{\mathrm{GHP}}^{k+1}\bigl( \bigl(T^n \bigl[\!\bigl[S^n_k\bigr]\!\bigr], \ell ^n_{k},
\bigl(s_{0}^n, \ldots,s_{k}^n
\bigr) \bigr), \bigl(\mathcal{T} {\llbracket S_k \rrbracket},
\ell_k,(s_0,\ldots,s_k) \bigr) \bigr) = 0.
\end{equation}
For each $n$ let $\mathcal{P}^n$ be a Poisson process on $T^n \times
[0,\infty
)$ with intensity measure \mbox{$\ell^n \otimes\ddd t$}. Then
$\mathcal{P}^n$ converges
in distribution to $\mathcal{P}$ in the sense of uniform convergence
on sets of
finite length measure \cite{DaVe2007}, Chapter~11.

Recall that we have enumerated the atoms of $\mathcal{P}$ as
$((p_i,\tau_i),i
\in\mathbb N)$; likewise, for each
$n \in\mathbb N$ we list the atoms of $\mathcal{P}^n$ as
$((p_{i}^n,\tau
_{i}^n),i \in
\mathbb N)$. We noted above that a.s. for each $i \in\mathbb N$, $p_i$
has degree
two in $\mathcal{T}$ and in $\mathcal{F}_{\tau_i-}$. Since $\mathcal
{T}$ is compact, yet
another application of Skorohod's theorem then implies that we may find
a space in which in addition to (\ref{eqskorohod2}) and (\ref
{eqskorohod3}), almost surely for each $\varepsilon> 0$ we have
%
\begin{equation}
\label{eqskorohod4} \lim_{n \to\infty} \inf \bigl\{i\dvtx\bigl|
\tau_{i}^n- \tau_i\bigr| > \varepsilon \bigr\} = \infty
\end{equation}
for each $k \geq0$ we have
%
\begin{eqnarray}\label{eqskorohod5}
\qquad && \lim_{n \to\infty} \inf \bigl\{i \in\mathbb N\dvtx \bigl|
\mathcal{T} {\llbracket S_k \rrbracket} \cap\{ p_{1},
\ldots,p_{i}\}\bigr| \neq\bigl|T^n \bigl[\!\bigl[S^n_k \bigr]\!\bigr]\cap \bigl\{p_{1}^n,\ldots,p_{i}^n
\bigr\}\bigr| \bigr\}
\nonumber\\[-8pt]\\[-8pt]
&&\qquad = \infty\nonumber
\end{eqnarray}
and for any fixed $k \geq0$, $i \geq1$, writing
\[
U_{k,i}^n= \bigl(s_{0}^n,
\ldots,s_{k}^n,p_{1}^n,
\ldots,p_{i}^n \bigr) \quad\mbox{and} \quad
U_{k,i}=(s_0,\ldots,s_k,p_1,
\ldots,p_{i}),
\]
we a.s. have
%
\begin{equation}
\label{eqskorohod6} d_{\mathrm{GHP}}^{k+1+i}\bigl( \bigl(T^n,d^n,
\mu ^n,U_{k,i}^n \bigr),(\mathcal{T},d,
\mu,U_{k,i}) \bigr) \to0
\end{equation}
as $n \to\infty$.

To sum up: by a sequence of applications of Skorohod's theorem we have
arrived at a space in which, after rescaling, the sequence $T^n$
converge almost surely to a Brownian CRT $\mathcal{T}$. We have additionally
coupled a sequence of random draws from the mass measure of $\mathcal
{T}$ to
its discrete counterpart, and a Poisson process on $\operatorname
{skel}(\mathcal
{T})\times
[0,\infty)$ to \emph{its} discrete counterpart, in such a way that any
finite collection of such points in the limiting space is arbitrarily
closely approximated by a corresponding (in both the informal and the
technical sense) collection of points in $T^n$, for $n$ large enough.
Furthermore, we have done so in such a manner that for any fixed $t>0$
and $k \geq1$, the operation of restricting the Poisson process to the
set of points arriving before time $t$ and falling within the subtree
spanned by the first $k$ random draws from the mass measure, commutes
with taking the large-$n$ limit.

\subsection{The convergence of the discrete backbones}\label{secconv-backbone}

In this section we continue to assume that $T^n$ is a conditioned
Galton--Watson tree with critical, finite-variance offspring
distribution $\xi$.
Before proving Theorem~\ref{teoexcbr}, we also need to express the
modified Aldous--Broder dynamics in the setting of conditioned
Galton--Watson trees. The only minor issue which needs to be addressed
is the fact that the modified Aldous--Broder dynamics should ignore
points of $\mathcal{P}^n$ which fall in an already cut subtree.

First, we consider the planted tree $T^n\langle\rho^n \rangle$ and
call the planted vertex $w^n$. We extend $\mu^n$ to $v(T^n\langle\rho
^n \rangle)$ by setting $\mu^n(\{w^n\})=0$. Recall the notation $a(v)$
for the parent of vertex $v$. For each $0 \le t \le\infty$, let
$T^n_t$ be the component of
\[
\bigl(v \bigl(T^n \bigl\langle\rho^n \bigr\rangle
\bigr), e \bigl(T^n \bigl\langle\rho^n \bigr\rangle \bigr)
\setminus \bigl\{ \bigl(p^n_i,a \bigl(p^n_i
\bigr) \bigr)\dvtx0 \le\tau_i^n \le t \bigr\} \bigr)
\]
containing $w^n$, and define $T^n_{t-}$ accordingly. (The forest in the
preceding equation is the finite-$n$ analogue of $\mathcal{F}^{\circ
}_t$, but
will not be used in what follows.) Write
\[
I^n_t = \bigl\{i \in\mathbb N\dvtx0 \le
\tau^n_i \le t, \mu^n \bigl(T^n_t
\bigr) < \mu ^n \bigl(T^n_{t-} \bigr) \bigr\}
\]
for the indices corresponding to ``effective'' cuts up to time $t$, and let
\[
\mathcal{P}^n_t = \bigl\{ p^n_i
\dvtx i \in I^n_t \bigr\}
\]
be the set of locations of these cuts. For $i \in I^n_{\infty}$ let
$x^n_i = p^n_i$ and let $y^n_i$ be the parent of $x^n_i$ in $T^n$ (here
we view $\rho^n$ as its own parent). Then, for $0 \le t \le\infty$, let
\[
\label{eqeffn} \mathcal{F}^n_t = \bigl(v
\bigl(T^n \bigr), e \bigl(T^n \bigr) \setminus \bigl
\{(x_i,y_i)\dvtx 0 \le\tau_i \le t \bigr\}
\bigr)
\]
and for $i \in I^n_{\infty}$ write $f^n_i$ for the component of
$\mathcal{F}
^n_{\infty}$ containing $x^n_i$. Note that $f^n_i$ is in fact a
component of $\mathcal{F}^n_t$ for all $\tau^n_i \le t \le\infty$.

Write $\kappa^n = |I^n_{\infty}|$, and write $\pi^n$ for the
permutation of $I^n_{\infty}$ that reorders the elements of
$I^n_{\infty
}$ in increasing order of the corresponding cut time, so that for $i,j
\in I^n_{\infty}$, $\pi^n(i) < \pi^n(j)$ if and only if $\tau^n_i <
\tau
^n_j$. Also, write
\[
u^n = x^n_{\pi^n(1)}\quad\mbox{and}\quad
v^n = x^n_{\pi^n(\kappa^n)}.
\]
Finally, let $\widehat{T}^n$ be the tree obtained from $\mathcal
{F}^n_{\infty}$ by
removing $w^n$, then adding the edges
\[
\bigl( x^n_{\pi^n(i+1)},x^n_{\pi^n(i)} \bigr)
\qquad1 \le i < \kappa^n.
\]
We view $\widehat{T}^n$ as rooted at $u^n$.

\begin{re*}
It is a standard fact that if $\xi$ is a
mean-one Poisson distribution (in fact, the mean does not matter), then
$T^n$ has the same distribution as the tree obtained from a uniform
Cayley tree on $[n]$ by removing the vertex labels. In this case,
Theorem~\ref{teokey} then implies that $\widehat{T}^n$ is distributed as a
uniform Cayley tree with labels removed, and $v^n$ is a uniformly
random element of~$v(\widehat{T}^n)$, independent of~$\widehat{T}^n$. This fact
will be used in the course of the proof of Theorem~\ref{teoexcbr} in
Section~\ref{secproof-transform}. However, it plays almost no role in
the current section. In particular, all results presented in this
section, with the exception of Corollary~\ref{corlocaltime}, are valid
for general critical, finite-variance conditioned Galton--Watson trees.
\end{re*}

%
\begin{lem}\label{lemmasslimit}
In the space where (\ref{eqskorohod2})--(\ref{eqskorohod6}) hold,
almost surely
\[
\bigl(\mu^n \bigl(T^n_t \bigr),t \geq0
\bigr) \to \bigl(\mu_t(\mathcal{T}_t),t \geq0 \bigr),
\]
in the sense of uniform convergence on compacts for the Skorohod $J_1$ topology.
\end{lem}
\begin{pf}
Write $\nu_k$ for the uniform measure on points $s_0,\ldots,s_k$. In
other words, given $\mathcal{T}$ and $s_0,\ldots,s_k$, $\nu_k$
assigns mass
$1/(k+1)$ to each of the points $s_0,\ldots,s_k$. Similarly, write
$\nu
_{k}^n$ for the uniform measure on $s_{0}^n,\ldots,s_{k}^n$.
By (\ref{eqskorohod6}), for any fixed $i,k \geq1$, almost surely
%
\begin{equation}
\label{eqmasslimit1} \lim_{n \to\infty}d_{\mathrm{GHP}}^{k+1+i}
\bigl( \bigl(T^n, d^n, \nu_k^n,U_{k,i}^n
\bigr),(\mathcal{T},d,\nu_k,U_{k,i}) \bigr) = 0.
\end{equation}
Also, by Theorem~8 of \cite{aldous1998standard}, for almost every
realization of $\mathcal{T}$,
%
\begin{equation}
\label{eqmasslimit2} \lim_{k \to\infty} \ddd_{\mathrm{P}}(
\nu_k,\mu)=0.
\end{equation}
(In fact, in \cite{aldous1998standard}, only almost sure weak
convergence is claimed, but the proof simply consists of an application
of the Glivenko--Cantelli theorem and is easily seen to yield
convergence with respect to $\ddd_{\mathrm{P}}$.)
Since for all $t \ge0$, $\mathcal{T}_t$ is a compact subspace of
$\mathcal{T}$, and
the $\mathcal{T}_t$ are decreasing in $t$, it follows that
%
\begin{equation}
\label{eqmasslimitk} \bigl(\nu_k(\mathcal{T}_t),t \ge0 \bigr)
\to \bigl( \mu_t(\mathcal{T}_t),t \ge0 \bigr)
\end{equation}
as $k \to\infty$.
Combining (\ref{eqmasslimit1}) with (\ref{eqskorohod4}) and (\ref
{eqskorohod5}), we obtain that for each $k \geq0$, almost surely
%
\begin{equation}
\label{eqmasslimit3} \bigl(\nu_k^n \bigl(T_t^n
\bigr),t \geq0 \bigr) \to \bigl(\nu_k(\mathcal{T}_t),t
\geq0 \bigr)
\end{equation}
as $n \to\infty$. Next, combining {(\ref{eqmasslimit2}) with (\ref
{eqskorohod6})}, we obtain that almost surely
\[
\lim_{k \to\infty} \lim_{n \to\infty} d_{\mathrm{GHP}}^{k+1+i}
\bigl( \bigl(T^n, d^n,\mu^n,U_{k,i}^n
\bigr),(\mathcal{T},d, \nu_k,U_{k,i}) \bigr) = 0,
\]
which together with (\ref{eqmasslimit1}) implies that almost surely
\[
\lim_{k \to\infty} \limsup_{n \to\infty}
\ddd_{\mathrm{P}} \bigl(\mu^n,\nu_k^n
\bigr)=0.
\]
In view of (\ref{eqmasslimitk}) and (\ref{eqmasslimit3}), this proves
the lemma.
\end{pf}
Next, for each $n\ge1$, reorder the elements of $\mathcal{P}^n$ as $\{
(p_{n,i},t_{n,i}),i \ge1\}$ so that $t_{n,i} < t_{n,i+1}$ for all $i
\ge1$.
We emphasize that here we consider all atoms of $\mathcal{P}^n$, not
only those
that correspond to ``effective cuts.''
%
\begin{lem}\label{lemmassapprox}
In the space where (\ref{eqskorohod2})--(\ref{eqskorohod6}) hold, a.s.
\[
\biggl( \frac{1} {\sigma\sqrt n}\sum_{\{j\dvtx t_{n,j} \le t\}}
\mu^n \bigl(T^n_{t_{n,j}} \bigr),t \geq0 \biggr) \to
\bigl( L(t), t \ge0\bigr),
\]
in the sense of uniform convergence on compacts for the uniform
distance, as \mbox{$n \to\infty$}.
\end{lem}
\begin{pf}
From Lemma~\ref{lemmasslimit} it is immediate that
\[
\biggl( \int_0^t \mu^n
\bigl(T_s^n \bigr)\,\ddd s,t \ge0 \biggr) \to
\biggl( \int_0^t \mu_s(
\mathcal{T}_{s}) \,\ddd s,t \ge0 \biggr)
\]
as $n \to\infty$. Also, $\ell^n(T^n) = \sigma\sqrt{n}$, so the set
$\{
\tau_i^n,i \in\mathbb N\}$
forms a Poisson point process of intensity $\sigma\sqrt{n}$ on
$[0,\infty)$, from which it follows straightforwardly that
\[
\biggl( \frac{1} {\sigma\sqrt n}\sum_{\{j\dvtx t_{n,j} \le t\}}
\mu^n \bigl(T^n_{t_{n,j}} \bigr),t \geq0 \biggr) \to
\biggl( \int_0^t \mu_s(\mathcal
{T}_s) \,\ddd s, t \ge0 \biggr)
\]
and the result then follows from (\ref{eqsamemasses}) and the
definition of $L(t)$ in (\ref{eqdeflocaltime}).
\end{pf}
Our next goal is to show that $(L(t),t \geq0)$ is the limit of the
discrete process which tracks the number of effective cuts up to time
$t\sqrt{n}$.
Write
\[
L^n(t) = \bigl|\mathcal{P}^n_t\bigr| = \# \bigl\{s
\leq t\dvtx \mu^n \bigl(T^n_{s} \bigr)<
\mu^n \bigl(T^n_{s-} \bigr) \bigr\}
\]
and note that, for every $n\ge1$, $L^n(t)$ increases to $\kappa^n(T^n)
= \#\{s> 0 \dvtx\mu^n(T^n_{s})< \mu^n(T^n_{s-})\}$, as $t\to\infty$.

%
\begin{teo}\label{teolocaltime}
In the space in which (\ref{eqskorohod2})--(\ref{eqskorohod6}) hold, a.s.
\[
\bigl(L^n(t)/ \bigl(\sigma n^{1/2} \bigr),t \ge0 \bigr) \to
\bigl(L(t),t \ge0 \bigr)
\]
in the sense of uniform convergence on compacts for the uniform
distance, as \mbox{$n \to\infty$}.
\end{teo}
In proving Theorem~\ref{teolocaltime} we will use a martingale
inequality from \cite{mcd}, Theorem~3.15.
Let $\{X_i\}_{i=0}^n$ be a bounded martingale with $X_0=0$, adapted to
a filtration $\{\mathcal{G}_i\}_{i=0}^n$.
Next let $V=\sum_{i=0}^{n-1} \mathbf{V}[X_{i+1}|\mathcal{G}_i]$, where
\[
\mathbf{V}[X_{i+1}|\mathcal{G}_i]:= {\mathbf E}
\bigl[(X_{i+1} - X_i)^2|\mathcal{G}_i
\bigr] = {\mathbf E}\bigl[X_{i+1}^2| \mathcal{G}_i
\bigr]-X_i^2
\]
is the predictable quadratic variation of $X_{i+1}$.
Define
\[
v=\operatorname{ess}\sup V \quad\mbox{and} \quad b = \max_{0 \leq i \leq n-1}
\operatorname{ess}\sup(X_{i+1}-X_i| \mathcal{G}_i),
\]
where for a random variable $X$, the essential supremum $\operatorname
{ess}\sup X$ is defined to equal $\sup\{x\dvtx{\mathbf P}(X \ge x) >
0\}$.
Then we have the following bound \cite{mcd}.
For any $t \geq0$,
%
\begin{equation}
\label{mcd} {\mathbf P}\Bigl(\max_{0 \leq i \leq n} X_i
\geq t\Bigr) \leq\exp \biggl( -\frac{t^2}{2v(1+bt/(3v))} \biggr).
\end{equation}
\begin{pf*}{Proof of Theorem~\ref{teolocaltime}}
In a first part, we prove uniform convergence on compacts for which we
do not need the trees $T^n$, $n\ge1$, to be uniform Cayley trees. Fix
$\delta> 0$ and $C > 0$.
By Lemma~\ref{lemmassapprox}, a.s.
\[
\sup_{0 \le t \le C} \biggl\llvert\frac{1} {\sigma\sqrt n}\sum
_{\{j\dvtx t_{n,j}
\le t\}} \mu^n \bigl(T^n_{t_{n,j-1}}
\bigr)- L(t) \biggr\rrvert\to0
\]
as $n \to\infty$. It follows that
%
\begin{eqnarray}
\label{eqlocaltime1} && {\mathbf P}\biggl(\limsup_{n \to\infty} \sup
_{0 \le t \le C} \biggl\llvert\frac{1} {\sigma\sqrt n}L^n(t)-L(t)
\biggr\rrvert> \delta\biggr)
\nonumber
\\[-8pt]
\\[-8pt]
&&\qquad\le{\mathbf P}\biggl(\limsup_{n \to\infty} \sup
_{0 \le t \le C} \biggl\llvert L^n(t)-\sum
_{\{j\dvtx t_{n,j} \le t\}} \mu^n \bigl(T^n_{t_{n,j-1}}
\bigr) \biggr\rrvert> \sigma\delta n^{1/2}\biggr).
\nonumber
\end{eqnarray}
Also, since $\mathcal{P}^n$ has intensity measure $\ell^n \otimes
\ddd t$ and
$\ell^n(v(T^n))=\sigma n^{1/2}$, we have that
\[
\lim_{x \to\infty} {\mathbf P}\Bigl(\liminf_{n \to\infty}
t_{n,\lfloor
x\sqrt{n}\rfloor} > C\Bigr)=1,
\]
which
implies that
the probability in (\ref{eqlocaltime1}) is at most
%
\begin{equation}
\lim_{x \to\infty} \label{eqlocaltime2} {\mathbf P}\biggl(\limsup
_{n \to
\infty} \max_{i \le x\sqrt{n}} \biggl\llvert
L^n(t_{n,i})- \sum_{1
\le j \le i}
\mu^n \bigl(T^n_{t_{n,j-1}} \bigr) \biggr\rrvert>
\sigma\delta n^{1/2}\biggr).
\end{equation}

For $i \geq1$, write
\begin{eqnarray*}
X_i & =& {\mathbf1}_{\{\mu^n  (T^n_{t_{n,i}}  ) < \mu^n
(T^n_{t_{n,i-1}}  )\}}.
\end{eqnarray*}
Also, for each $i \geq1$, let $\mathcal{P}^{n}_i = \{p_{n,1},\ldots,p_{n,i}\}$.
Taking $\mathcal{G}_{n,i}$ to be the sigma field generated by $T^n$
and $\mathcal{P}
^{n}_i$, then
$(X_i,i \geq1)$ is adapted to $(\mathcal{G}_{n,i},i \geq1)$.
Note that
\begin{eqnarray*}
{\mathbf E}[X_i|\mathcal{G}_{n,i-1}] & =& \mu^n
\bigl(T^n_{t_{n,i-1}} \bigr) = {\mathbf E}\bigl[X_i^2|
\mathcal{G}_{n,i-1}\bigr],
\end{eqnarray*}
so in all cases $\V{X_i|\mathcal{G}_{n,i-1}} \le1/4$.
Also, for all $i \geq1$ we have $\sum_{j=1}^i X_j = L^n(t_{n,i})$.
By (\ref{mcd}), for any fixed $x > 0$ and $n \ge1$, we thus have
%
\begin{eqnarray}
\label{eqxhatmart}
&& {\mathbf P}\biggl(\max_{i \le x\sqrt{n}} \biggl\llvert
\sum_{j\le i} X_j - \sum
_{j \leq i} \mu^n \bigl(T^n_{t_{n,j-1}}
\bigr) \biggr\rrvert\geq y\biggr)
\nonumber\\[-8pt]\\[-8pt]
&&\qquad \le2\exp\biggl( -\frac{2y^2}{x\sqrt {n}(1+4y/(3x\sqrt{n}))} \biggr).\nonumber
\end{eqnarray}
Applying this bound with $y=\delta\sqrt{n}$ and summing over $n$, it
follows by Borel--Cantelli that
\[
{\mathbf P}\biggl(\limsup_{n \to\infty} \max_{i \le x\sqrt{n}}
\biggl\llvert L^n(t_{n,i})-\sum
_{1 \le j \le i} \mu^n \bigl(T^n_{t_{n,j-1}}
\bigr) \biggr\rrvert> \delta n^{1/2}\biggr)=0,
\]
which together with (\ref{eqlocaltime2}) shows that $(L^n(t)/(\sigma
n^{1/2}),0 \le t \le C) \to(L(t), 0 \le t \le C)$ almost surely for the
uniform distance.
\end{pf*}
%
%
\begin{cor}\label{corlocaltime}
If $\xi$ is the $\operatorname{Poisson}(1)$ distribution then in the space in which
(\ref{eqskorohod2})--(\ref{eqskorohod6}) hold,
\[
\bigl(L^n(t)/n^{1/2},t \ge0 \bigr) \to \bigl(L(t),t \ge0
\bigr)
\]
in probability in the sense of uniform convergence on $[0,\infty)$.
\end{cor}
\begin{pf}
If $\xi$ is the $\operatorname{Poisson}(1)$ distribution, then $\sigma=1$. Uniform
convergence on compacts follows from Theorem~\ref{teolocaltime}.
Furthermore, as noted in the remark just before Lemma~\ref
{lemmasslimit}, in this case $T^n$ is distributed as a uniform Cayley
tree on $[n]$ with labels removed. Also, $\widehat{T}^n$ is again
distributed as a uniform Cayley tree with labels removed, and $\kappa
^n(T^n)$ is the distance between $u^n$ and $v^n$ in $\widehat{T}^n$, it
follows from Theorem~\ref{teokey} that $\kappa^n(T^n)/n^{1/2}$
converges in distribution to a Rayleigh random variable.

For any $t,\delta> 0$, given that $\mu^n(T^n_{t}) \le\delta$, the
difference $\kappa^n(T^n)-L^n(t)$ is dominated by the number of cuts
required to isolate the root of a uniform Cayley tree on $\lfloor
\delta n \rfloor$ vertices. It follows that for any $\varepsilon> 0$,
%
\begin{equation}
\label{eqforpal} \lim_{t \to\infty} \limsup_{n \to\infty} {
\mathbf P}\bigl(\kappa^n \bigl(T^n \bigr) -
L^n(t) > \varepsilon n^{1/2}\bigr) = 0.
\end{equation}
By the principle of accompanying laws (Theorem 9.1.13 of \cite
{stroock}), in the space in which (\ref{eqskorohod2})--(\ref
{eqskorohod6}) hold, we have
\[
\frac{\kappa^n(T^n)}{n^{1/2}} \stackrel{p}{\rightarrow} L(\infty)=\lim_{t\to\infty} L(t),
\]
which together with (\ref{eqforpal}) implies uniform convergence on
$[0,\infty)$. [This also yields a second proof that $L(\infty)$ is
almost surely finite, as promised just after (\ref{eqdeflocaltime}).]
\end{pf}
Before proving Theorem~\ref{teoexcbr} we note one consequence of
Corollary~\ref{corlocaltime}, stated
in the \hyperref[secintro]{Introduction} as Corollary~\ref{corrayleigh}. A different proof
of this result can be found in Abraham and Delmas \cite{AbDe2013b}.
%
\begin{pf*}{Proof of Corollary~\ref{corrayleigh}}
In proving Corollary~\ref{corlocaltime} we showed the existence of a
space in which
\[
L(\infty) \stackrel{p} {=} \lim_{n \to\infty}
\frac{\kappa
^n(T^n)}{n^{1/2}}
\]
and the latter limit is Rayleigh distributed by Theorem~\ref{teokey}
The lemma then follows from the definition of $L(t)$ in (\ref
{eqdeflocaltime}) and (\ref{eqsamemasses}).
\end{pf*}

\subsection{The proof of Theorem~\texorpdfstring{\protect\ref{teoexcbr}}{5.1}}\label
{secproof-transform}

In this section, in order to use the discrete results of Section~\ref
{seccayley}, we assume that $\xi$ is the $\operatorname{Poisson}(1)$ distribution,
or equivalently (see the remark just before Lemma~\ref{lemmasslimit})
that $T^n$ is a uniform Cayley tree on $[n]$ with its labels removed.
In particular, this implies that $\sigma=1$.

Recall the definitions of the trees $\{f_i,i \in I_{\infty}\}$ and $\{
f^n_i,i \in I^n_{\infty}\}$ from pages~\pageref{eqeffi} and~\pageref{eqeffn} (here we simply view each $f_i$ as a subset of $\mathcal
{T}$). Also,
write $\hat{d}_n$ for $n^{-1/2}$ times the standard graph distance on
$\widehat{T}^n$, and write $\hat{\mu}^n$ for the uniform probability
measure on $v(\widehat{T}^n)$.

We work in a space where (\ref{eqskorohod2})--(\ref{eqskorohod6}) all hold.
For any $\varepsilon> 0$, let
\[
J_{\varepsilon} = \bigl\{i \in I_{\infty}\dvtx\mu_{\infty}(f_i)
> \varepsilon \bigr\}.
\]
The set $J_{\varepsilon}$ is necessarily finite (it has size at most
$\varepsilon
^{-1}$), so $K^{(\varepsilon)}:= \sup\{i\dvtx i \in J_{\varepsilon}\}$ is
a.s. finite. By
(\ref{eqskorohod5}), for all $n$ sufficiently large we in particular
have that $J_{\varepsilon} \subset I^n_{\infty}$, and we hereafter
assume that
inclusion indeed holds.

Let $S = \{u,v\} \cup\bigcup_{i \in J_{\varepsilon}}f_i$, and let $\widehat
{T}_{\varepsilon} = \bigcup_{x,y \in S} {\llbracket x,y \rrbracket}$.
In words, $\widehat{T}_{\varepsilon}$ is the minimal subtree of $\widehat{T}$ which
contains each of the subtrees $f_i$, $i \in J_{\varepsilon}$ and also contains
the distinguished nodes $u$ and $v$. Likewise,
let
\[
\widehat{T}^n_{\varepsilon} = \widehat{T}^n \biggl[\!\!\biggl[ \bigl
\{u^n,v^n \bigr\} \cup\bigcup
_{i
\in J_{\varepsilon}} v \bigl(f^n_i \bigr)
\biggr]\!\!\biggr].
\]
We let $\hat{d}_{\varepsilon} = \hat{d}|_{\widehat{T}_{\varepsilon}}$, and
define $\hat{\mu}_{\varepsilon},\hat{d}^n_{\varepsilon},\hat{\mu}^n_{\varepsilon}$
accordingly.

The set $I_{\infty}$ is countable and $J_{\varepsilon} \uparrow
I_{\infty}$ as
$\varepsilon\downarrow0$. Also, it follows from the result of Aldous and
Pitman \cite{aldous1998standard} mentioned earlier that $\sum_{i \in
I_{\infty}} \mu_{\infty}(f_i)=1$ a.s., and we thus a.s. have
\[
\lim_{\varepsilon\downarrow0} \sum_{i \notin J_{\varepsilon}}
\mu_{\infty}(f_i) = 0.
\]
Since $\mathcal{T}$ is compact and each $f_i$ can be viewed as a
subtree of $\mathcal{T}
$, we must also a.s. have
\[
\lim_{\varepsilon\downarrow0} \sup_{i \notin J_{\varepsilon}}
\operatorname{diam}(f_i) = 0.
\]
(Otherwise, there would exist $\delta> 0$ and an infinite set $S
\subset I_{\infty}$ such that
for each $i \in S$, $f_i$ has height greater than $\delta$. For $i \in
S$, letting $q_i$ be any point in $f_i$
whose distance to the root $p_i$ of $f_i$ is at least $\delta$, the set
$\{q_i,i \in S\}$ is infinite and its elements have pairwise distance
at least $\delta$, contradicting compactness.)
By these facts and by (\ref{eqskorohod6}), for any $\delta> 0$ there
is $N=N(\varepsilon,\delta)$ which is almost surely finite, such that
for all
$n \ge N$ and $i\in J_\varepsilon$,
%
\begin{equation}
\label{eqfinal0} d_{\mathrm{GHP}}^{k+1+i}\bigl( \bigl(T^n,d^n,
\mu ^n,U_{k,i}^n \bigr),(\mathcal{T},d,
\mu,U_{k,i}) \bigr) < \delta
\end{equation}
and additionally $\sum_{i \notin J_{\varepsilon}} \mu_{\infty}(f_i) <
\delta$
and $\sup_{i \notin J_{\varepsilon}} \operatorname{diam}(f_i) < \delta
$. We fix a
correspondence $C \in\mathscr{C}((T^n,d^n,\mu^n,U_{k,i}^n),(\mathcal
{T},d,\mu,U_{k,i}))$ with $\operatorname{dis}(C) < 2\delta$ and containing the
appropriate
pairs of points from $U_{k,i}^n$ and $U_{k,i}$.
It follows from the fact that $\sup_{i \notin J_{\varepsilon}}
\operatorname{diam}(f_i) <
\delta$ that
%
\begin{equation}
\label{eqfinal1} d_{\mathrm{GHP}}^2 \bigl( \bigl(\widehat{T},\hat{d},
\hat{\mu},(u,v) \bigr), \bigl(\widehat{T}_{\varepsilon},\hat{d}_{\varepsilon},\hat{\mu}_{\varepsilon},(u,v) \bigr) \bigr) < \delta
\end{equation}
and that
%
\begin{equation}
\label{eqfinal2} \sup_{i \in I^n_{\infty}\setminus J_{\varepsilon} } n^{-1/2}
\operatorname{diam}\bigl(f^n\bigr) < 3 \delta.
\end{equation}
Next, write $m_{\delta} = \sup_{x \in\mathcal{T}} \mu(B(x,\delta
))$, where
$B(x,\delta)$ is the ball of radius $\delta$ around $x$ in $\mathcal
{T}$. We
have $m_{\delta} \downarrow0$ a.s. as $\delta\to0$. Choose $0 <
\delta< \varepsilon^2$ small enough that $m_{4\delta} < \varepsilon^2$.
Then for $n \ge N(\varepsilon,\delta)$, and for all $i \in J_{\varepsilon
}$, by
considering the $\delta$ blow-up $C_{\delta}$ of the correspondence
$C$, we see that
%
\begin{equation}
\label{eqfinal3} d_{\mathrm{GHP}}^{1}\bigl( \bigl(f^n_i,d^n_{\infty
}|_{f^n_i},
\mu^n_{\infty}|_{f^n_i} \bigr),(f_i,d_{\infty}|_{f_i},
\mu_{\infty}|_{f_i}) \bigr) < 2\delta+m_{4\delta} < 2
\varepsilon^2.
\end{equation}
In particular, for each $i \in J_{\varepsilon}$, $|\mu^n_{\infty
}(f^n_i) - \mu
_{\infty}(f_i)| < 2\varepsilon^2$, so
%
\begin{equation}
\label{eqfinal4} \sum_{i \in J_{\varepsilon}} \bigl\llvert
\mu^n_\infty \bigl(f^n_i \bigr) -{
\mu_\infty}(f_i) \bigr\rrvert< 2\varepsilon^2
|J_\varepsilon| < 2 \varepsilon
\end{equation}
and
%
\begin{equation}
\label{eqfinal5} \sum_{i \in I^n_{\infty}\setminus J_{\varepsilon}} {\mu^n_\infty}
\bigl(f^n_i \bigr) \le2\varepsilon+ \delta< 3\varepsilon.
\end{equation}
By (\ref{eqfinal2}) and (\ref{eqfinal5}), it follows that for all $n$
sufficiently large,
\[
d_{\mathrm{GHP}}^2 \bigl( \bigl(\widehat{T}^n,
\hat{d}^n, \hat{\mu}^n, \bigl(u^n,v^n
\bigr) \bigr), \bigl( \widehat{T}^n_{\varepsilon},\hat{d}^n_{\varepsilon},
\hat{\mu}^n_{\varepsilon}, \bigl(u^n,v^n
\bigr) \bigr) \bigr) < 3(\delta+ \varepsilon) < 6 \varepsilon.
\]
For each $i \in I_{\infty}$, $L(\tau_i) = \hat{d}(u,x_i)$ and for each
$i \in I^n_{\infty}$, $n^{-1/2}L^n(\tau_i^n) = \hat{d}(u^n,x^n_i)$. By\vspace*{2pt}
Corollary~\ref{corlocaltime}, it follows that for all $i \in
J_{\varepsilon
}$, for all $n$ sufficiently large, $|\hat{d}(u,x_i) - \hat{d}^n(u^n,x_i^n)| < \delta$.
Together with (\ref{eqfinal3}) and (\ref{eqfinal4}), this implies that
\[
d_{\mathrm{GHP}}^2 \bigl( \bigl(\widehat{T}^n_{\varepsilon},
\hat{d}^n_{\varepsilon},\hat{\mu}^n_{\varepsilon},
\bigl(u^n,v^n \bigr) \bigr), \bigl(\widehat{T}_{\varepsilon},
\hat{d}_{\varepsilon},\hat{\mu}_{\varepsilon},(u,v) \bigr) \bigr) < \max \bigl(
\delta+2\varepsilon^2,2\varepsilon \bigr) < 3\varepsilon.
\]
By the two preceding inequalities, (\ref{eqfinal1}) and the triangle
inequality, we obtain that a.s. for all $n$ sufficiently large,
\[
d_{\mathrm{GHP}}^2 \bigl( \bigl(\widehat{T}^n,
\hat{d}^n, \hat{\mu}^n, \bigl(u^n,v^n
\bigr) \bigr), \bigl( \widehat{T},\hat{d},\hat{\mu},(u,v) \bigr) \bigr) < 9\varepsilon +
\delta< 10\varepsilon.
\]
Since $\varepsilon> 0$ was arbitrary, the first assertion of the theorem then
follows from Theorem~\ref{teokey}.

Finally, since the distribution of the collection $\{y_i,i \in
I_{\infty
}\}$ is determined by its finite-dimensional distributions, the
assertion in the statement of Theorem~\ref{teoexcbr} about the
collection $\{y_i,i \in I_{\infty}\}$ then follows from Lemma~\ref
{lemcutonk}, below, whose straightforward proof is omitted.

%
\begin{lem}\label{lemcutonk}
Fix $n \ge1$, $k \ge1$, let $K = \{i \in I^n_{\infty}\dvtx p^n_i
\in
T^n{\llbracket S^n_k \rrbracket}\}$ and let $j \in K$ be the element
$i \in K$ which
minimizes $\tau^n_{i}$. Suppose that $T^n$ is a uniform Cayley tree on
[n]. Then for any $S \subset v(T^n)$, any tree $t$ with $v(t)=S$, and
any $y \in S$,
\[
{\mathbf P}\bigl(T^n_{\tau^n_{j}}=t\mbox{ and }
y^n_{j}=y | v \bigl(T^n_{\tau^n_{j}}
\bigr) = S\bigr) = |S|^{-|S|}.
\]
\end{lem}

\section{Conditioned Galton--Watson trees with finite variance}\label
{secsigma}

We now want to prove that the picture that we have obtained for the
process in the case of uniform Cayley trees is also valid when one
considers conditioned Galton--Watson trees with critical, finite-variance offspring distribution.
Fix an offspring distribution $\xi=(\xi_0,\xi_1,\ldots)$ with
\[
\sum_{i\ge0} i\xi_i=1\quad\mbox{and}\quad\sum_{i\ge0} i(i-1) \xi_i=
\sigma^2 \in(0,\infty).
\]

%
\begin{teo}\label{teogwcut} Let $T^n$ be distributed as a
Galton--Watson tree with offspring distribution $\xi$, conditioned to
have $n$ vertices. Then after rescaling, the number of cuts $\kappa
(T^n)$ required to isolate the root of $T^n$ is asymptotically Rayleigh
distributed,
\[
\lim_{n\to\infty} {\mathbf P}\bigl(\kappa \bigl(T^n
\bigr)\ge\sigma x \sqrt n\bigr) = e^{-x^2/2}.
\]
\end{teo}

Under a finite-variance assumption, Galton--Watson trees conditioned on
their size have the same scaling limit as uniform Cayley trees, so when
looking at a $(n,\sqrt n)$ rescaling for time and space, the cutting
process will essentially look the same. Completing the argument then
boils down to showing that once the left-over tree has size $o(n)$ the
number of cuts needed to completely destroy it is $o(\sqrt n)$. The
following lemma shows that this is indeed the case. (Although the
factor $\varepsilon^{1/6}$ is certainly not best possible, it is
sufficient for our needs.)

%
\begin{lem}\label{lemtightnessgw}
Suppose that ${\mathbf E}\xi=1$ and $\V\xi
=\sigma^2 \in(0,\infty)$. Let $T^n$ be a Galton--Watson tree with
progeny distribution $\xi$, conditioned on having size $n$. Let also
$\tau^n(\varepsilon)=\inf\{t\dvtx\mu^n(T^n_t)<\varepsilon\}$. Then
\[
\limsup_{n\to\infty}{\mathbf P}\bigl(\kappa \bigl(T^n_{\tau
^n(\varepsilon)}
\bigr)\ge\varepsilon^{1/6} \sqrt n\bigr) \mathop{\rightarrow}_{\varepsilon\to
0} 0.
\]
\end{lem}
\begin{pf}
Recall that for a rooted tree $T$ and a node $v$ of $T$, we write
$h_T(v)$ for the \emph{height} of $v$ in $T$, which is the number of
edges on the path from the root to $v$. We also write $h(T)= \max_{v
\in v(T)} h_T(v)$, and call $h(T)$ the height of $T$. Finally,
for $i \ge0$ write $w_i(T) = \#\{v \in v(T)\dvtx h_T(v)=i\}$.

For any $x,y>0$ we have
\begin{eqnarray}
\label{equseheight} {\mathbf P}\bigl(\kappa \bigl(T^n_{\tau^n(\varepsilon
)}
\bigr) \ge y \sqrt n\bigr) &\le& {\mathbf P}\bigl(\kappa \bigl(T^n_{\tau
^n(\varepsilon)}
\bigr) \ge y \sqrt n,h \bigl(T_{\tau
^n(\varepsilon)}^n \bigr) \le x\sqrt n
\bigr)
\nonumber
\nonumber
\\[-8pt]
\\[-8pt]
&&{} + {\mathbf P}\bigl(h \bigl(T_{\tau^n(\varepsilon)}^n \bigr)> x \sqrt
n\bigr).
\nonumber
\end{eqnarray}

The first term above is easily bounded using Markov's inequality. We
use Janson's representation of the number of cuts as records in the
tree \cite{Janson2004,Janson2006c}. Given a tree $t$, rooted at $r$,
one can assign extra labels to the vertices using a random permutation
of $\{1,2,\ldots, |t|\}$. This random permutation determines the order
in which the vertices are considered for cutting. In this
representation, a vertex $u$ will actually produce a cut if and only if
the path $\llbracket r,u\rrbracket$ between $u$ and the root has not
been previously cut. This happens precisely if $u$ has the minimum
label of all vertices on $\llbracket r,u\rrbracket$. In particular,
conditional on the height $h_t(u)$ of $u$ in $t$, the probability that
a vertex $u$ produces a cut is $(h_t(u)+1)^{-1}$. It follows that
%
\begin{eqnarray} \label{eqcgw1st}
&&{\mathbf P}\bigl(\kappa \bigl(T^n_{\tau^n(\varepsilon)} \bigr) \ge y
\sqrt n,h \bigl(T_{\tau^n(\varepsilon)}^n \bigr) \le x\sqrt n\bigr)
\nonumber
\\
&&\qquad \le\frac{1}{y\sqrt n} \cdot{\mathbf E}\bigl[\kappa \bigl(T^n_{\tau
^n(\varepsilon)}
\bigr) {\mathbf1}_{\{h  (T_{\tau
^n(\varepsilon)}^n  ) \le x\sqrt n\}}\bigr]
\nonumber
\\
&&\qquad \le\frac{1} {y \sqrt n} \cdot{\mathbf E}\biggl[\sum
_{u\in T^n_{\tau
^n(\varepsilon)}} \frac{1} {1+h_{T^n}(u)} {\mathbf1}_{\{h
(T_{\tau^n(\varepsilon)}^n  ) \le x\sqrt n\}}\biggr]
\nonumber\\[-8pt]\\[-8pt]
&&\qquad \le\frac{1} {y \sqrt n} \cdot{\mathbf E}\biggl[\sum
_{0 \le i \le x
\sqrt{n}} \sum_{\{u\dvtx h_{T^n}(u)=i\}}
\frac{1} {1+h_{T^n}(u)} {\mathbf1}_{\{h  (T_{\tau^n(\varepsilon)}^n  ) \le x\sqrt
n\}}\biggr]
\nonumber
\\
&&\qquad \le\frac{1} {y \sqrt n} \cdot\sum_{0\le i \le x \sqrt n}
\frac{{\mathbf E}[w_i(T^n)]}{1+i}
\nonumber
\\
&&\qquad \le\frac{1} {y \sqrt n} \cdot\sum_{0\le i \le x \sqrt n}
\frac{C
i}{1+i} \le\frac{C x} y,\nonumber
\end{eqnarray}
we used the fact that ${\mathbf E}[w_k(T^n)]\le C k$ uniformly in $k\ge
0$ and
$n\ge0$ (see Devroye and Janson \cite{DeJa2011a}) to obtain the
second-to-last inequality.

To bound the second term, we relate the finite-$n$ trees $T^n$ to their
limit $\mathcal{T}$.
We work in a space in which (\ref{eqskorohod2})--(\ref{eqskorohod6})
all hold, and recall from Section~\ref{seccouplings} the definitions
of the collections of points $(s_i,i \ge1)$ and $\{p_i^n\dvtx i \in
\mathbb N\}$,
and of their finite-$n$ counterparts $(s^n_i, i \ge1)$ and $\{
p_i^n\dvtx
i\in\mathbb N\}$.
In particular, recall the definitions of the sequences $S_k$, $S^n_k$,
from page~\pageref{eqskorohod3}.

We now use that for all $\delta> 0$,
\[
\lim_{k \to\infty}{\mathbf P}\bigl(d_{\mathrm{GH}}^1\bigl((
\mathcal {T},d,\rho),(\mathcal{T} {\llbracket S_k
\rrbracket},d|_{\mathcal
{T}{\llbracket S_k \rrbracket}},\rho) > \delta\bigr)\bigr) = 0.
\]
By (\ref{eqskorohod2}), we then also have that
\[
\lim_{k \to\infty} \limsup_{n \to\infty} {\mathbf
P}\bigl(d_{\mathrm
{GH}}\bigl( \bigl(T^n,d^n,
\rho^n \bigr), \bigl(T^n \bigl[\!\bigl[ S_k^n\bigr]\!\bigr],d^n|_{T^n{\llbracket S_k^n \rrbracket}}, \rho^n \bigr) > \delta
\sqrt n\bigr)\bigr) = 0.
\]
Equations (\ref{eqskorohod4}), (\ref{eqskorohod5}) and (\ref
{eqskorohod6}) provide a coupling of the cuts falling on $T^n
{\llbracket S_k^n \rrbracket}$
with those falling on $\mathcal{T}{\llbracket S_k \rrbracket}$ so
that for any fixed $t > 0$ and
for all sufficiently large $n$, the cuts falling within
$T^n {\llbracket S_k^n \rrbracket}$ and within $\mathcal
{T}{\llbracket S_k \rrbracket}$ occur at essentially the
same times and at essentially the same locations.
[This is precisely formalized by (\ref{eqskorohod4}), (\ref
{eqskorohod5}) and (\ref{eqskorohod6}).]
It then follows that in this space, for any $\varepsilon> 0$ and $\delta
> 0$,
%
\[
\limsup_{n \to\infty} {\mathbf P}\bigl(d_{\mathrm{GH}}^1
\bigl( \bigl(T^n_{\tau^n(\varepsilon)}, \sigma n^{-1/2}d^n|_{T^n_{\tau^n(\varepsilon
)}},
\rho^n \bigr),(\mathcal{T}_{\tau(\varepsilon)},d|_{\mathcal{T}
_{\tau(\varepsilon)}},\rho)
\bigr) > \delta\bigr) =0.
\]
Taking $\delta= x\sqrt{\varepsilon}$, from this we immediately obtain that
%
\begin{eqnarray}
\label{eqcgw2d} \limsup_{n \to\infty} {\mathbf P}\biggl(h
\bigl(T^n_{\tau^n(\varepsilon
)} \bigr) \ge\frac{ x}\sigma\sqrt{
\varepsilon n} \biggr) &\le&{\mathbf P}\bigl(h( \mathcal{T}_{\tau(\varepsilon)})\ge x
\sqrt\varepsilon\bigr)
\nonumber\\[-8pt]\\[-8pt]
& \le& e^{-\alpha x^2}\nonumber
\end{eqnarray}
for some constant $\alpha>0$.
The last inequality holds since: conditional on its mass, $\mathcal
{T}_{\tau
(\varepsilon)}$ is a Brownian CRT (see \cite{aldous1998standard}, equation
(44)); we have $\mu(\mathcal{T}_{\tau(\varepsilon)})\le\varepsilon
$; the height of a
Brownian CRT is distributed as the supremum of a Brownian excursion;
and the supremum of a Brownian excursion has Gaussian tails~\cite{Kennedy1976}.

Then, choosing, for instance, $x=\varepsilon^{1/3}$ {in (\ref
{equseheight})} and
$y=\varepsilon^{1/6}$ and using the bounds in (\ref{eqcgw1st}) and
(\ref{eqcgw2d}) to bound
(\ref{equseheight}) proves the result.
\end{pf}

Putting together Corollary~\ref{corrayleigh} and the following lemma
then yields Theorem~\ref{teogwcut}.

%
\begin{lem}\label{lemgw}
Let $T^n$ be a Galton--Watson tree with
offspring distribution $\xi$ conditioned to have size $n$, and let
$\mathcal T$ be a Brownian CRT. If ${\mathbf E}\xi=1$ and
$\operatorname{Var}(\xi
)=\sigma^2\in
(0,\infty)$ then
\[
\frac{\kappa(T^n)}{\sigma\sqrt n}\mathop{\rightarrow}_{n\to\infty}^{d}\int_{0}^\infty
\mu(\mathcal T_t) \,dt.
\]
\end{lem}
\begin{pf}
Write
$T^n_t$ for the subtree containing the root at time $t$ of the cutting
process, and
as in Section~\ref{secconv-backbone} write
\[
L^n(t)=\# \bigl\{s\le t\dvtx\mu^n \bigl(T^n_t
\bigr)<\mu^n \bigl(T^n_{t-} \bigr) \bigr\}
\]
for the number of cuts occurring before time $t$, Theorem~\ref
{teolocaltime} implies that for any fixed $t\in[0,\infty)$
%
\begin{equation}
\label{eqlimln} \frac{L^n(t)}{\sigma\sqrt n} \stackrel{d} {\to}\int_0^t
\mu(\mathcal T_{t}) \,\ddd t
\end{equation}
as $n \to\infty$.\vadjust{\goodbreak}

Recall that $\tau^n(\varepsilon)=\inf\{t\dvtx\mu
^n(T^n_t)<\varepsilon\}$. Since
$\tau^n(\varepsilon)<\infty$ almost surely and additionally $\tau
^n(\varepsilon)\to
\tau(\varepsilon)$ in distribution jointly with the convergence in
(\ref{eqlimln}),
we have
\[
\frac{L^n(\tau^n(\varepsilon))}{\sigma\sqrt n}\mathop{\rightarrow}^{d} \int_0^{\tau
(\varepsilon)}
\mu(s)\,ds.
\]
On the other hand,
\[
\frac{\kappa(T^n)-L^n(\tau^n(\varepsilon))}{\sqrt n} \le\frac
{\kappa
(T^n_{\tau^n(\varepsilon)})}{\sqrt n} \mathop{\rightarrow}_{\varepsilon\to 0}0,
\]
in probability, uniformly for all $n$ sufficiently large, by Lemma~\ref
{lemtightnessgw}.
Since \mbox{$\tau(\varepsilon)\to\infty$} almost surely as $\varepsilon
\to0$, it
follows that
\[
\frac{\kappa(T^n)}{\sigma\sqrt n} \stackrel{d} {\to}\int_0^\infty
\mu( \mathcal T_{t})\,\ddd t
\]
as $n \to\infty$, as claimed.
\end{pf}

\begin{appendix}
\section{Corollary~\texorpdfstring{\lowercase{\protect\ref{corbert}}}{1.2}: Proof sketch and discussion}\label{appendixcor}
Fix a rooted tree $t$ with nodes $\{1,\ldots,n\}$, and a sequence
$(v_1,\ldots,v_k) \in\{1,\ldots,n\}^k$ of nodes of $t$. In $t$, view
the children of a node as ordered so that node labels increase from
left to right. Let $t'$ be the subtree of $t$ spanned by the root and
$v_1,\ldots,v_k$. Let the \emph{reduced} tree $t^*$ be obtained from
$t'$ by suppressing degree-two vertices (so in $t^*$, the parent of
$v_i$ corresponds to the most recent common ancestor of $v_i$ and any
of the $v_j$ with $v_j \ne v_i$) and suppressing vertex labels (but
keeping the plane tree structure). Since $t^*$ has no nodes of degree
2, it has at most $2k-1$ edges, with equality precisely if it is binary
and $v_1,\ldots,v_k$ are distinct. Write $e$ for the number of edges
of $t^*$.

Given the tree $t^*$, one may recover $t$ by listing an ordered rooted
forest $f_1,\ldots,f_m$, together with a weak composition $(c_1,\ldots,c_e)$ of $m$ into $e$ parts. To do so, list the edges of $t^*$
according to their order of first traversal by a contour exploration of
$t^*$. Then glue the roots of $f_1,\ldots,f_{c_1}$ along the first
edge, $f_{c_1+1},\ldots,f_{c_1+c_2}$ along the second edge, and so on.
A result of Riordan \cite{MR0228386} states that the number of ordered
rooted forests on vertices $\{1,\ldots,n\}$ with $m$ components is
\[
B_{n,m}:= m!\pmatrix{n-1 \cr m-1} n^{n-m}.
\]
It follows that the number of trees $t$ with reduced tree $t^*$ and
such that $t'$ has $m$ vertices, is
\[
A_{k} \bigl(t^* \bigr) B_{n,m} \pmatrix{m + e - 1
\cr e -
1},
\]
where ${m+e-1\choose e-1}$ is the number of weak compositions of $m$
into $e$ parts,
and $A_k(t^*)$ is a combinatorial factor counting the possible
locations of $v_1,\ldots,v_k$ in $t^*$.
More precisely, $A_k(t^*)$ is the number of multi-sets of vertices from
$t^*$ of size $k$ (with multiplicity) containing all leaves of $t^*$.
In particular, if $t^*$ has $k$ leaves then $A_k(t^*)=1$. Since the
total number of $k$-marked rooted trees on $[n]$ is $n^{n+1-k}$, and
the number of binary plane trees with $k$ leaves is given by the
$(k-1)$'st Catalan number, straightforward approximations then prove
Corollary~\ref{corbert}.

It seems worthwhile to further observe that for any $p \in[1,\infty)$,
the collection of laws of the random variables
$((n^{-1/2}M(T_n,S_k))^p,n \ge1)$ forms a uniformly integrable family.
To see this, using the notation of Theorem~\ref{teokcoup}, let $E_n$
be number of edges in the subtree of $T_n$ spanned by its root
$r=r(T_n)$ plus $V_1,\ldots,V_k$. By Theorem~\ref{teokcoup},
\[
M(T_n,S_k) \stackrel{d} {=}E_n + k.
\]
Furthermore, writing $d_n$ for graph distance in $T_n$, we have
\begin{eqnarray*}
E_n+k & \le& \sum_{i=1}^k
d_n(r,V_i) + k. 
\end{eqnarray*}
Since $(d_n(r,V_i),1 \le i \le k)$ are i.i.d. it follows by a union
bound that
for $x> 0$,
%
\begin{eqnarray}
\label{eqlpbound0} {\mathbf P}\bigl( \bigl(n^{-1/2} M(T_n,S_k)
\bigr)^p \ge x\bigr) & \le& k {\mathbf P}\bigl(d_n(r,V_1)
\ge n^{1/2}x^{1/p}/k - 1\bigr).
\end{eqnarray}
But the law of $d_n(r,V_1)$ is well-known (see \cite{MR0263685} for an
early derivation): we have, for $\ell\ge1$,
%
\begin{equation}
\label{eqdndistr} {\mathbf P}\bigl(d_n(r,V_1) \ge\ell
\bigr) = \prod_{j=1}^{\ell-1} \biggl(1-
\frac{j}{n} \biggr) \le \exp \biggl(-\frac{\ell(\ell-1)}{2n} \biggr).
\end{equation}
Using (\ref{eqlpbound0}) and (\ref{eqdndistr}), standard
manipulations imply that for all $p\ge1$,
%
\begin{eqnarray*}
&& \lim_{K \to\infty} \sup_{n\ge1} {\mathbf E}\bigl[
\bigl(n^{-1/2} M(T_n,S_k) \bigr)^p {
\mathbf1}_{\{ (n^{-1/2} M(T_n,S_k)
)^p \ge K\}}\bigr]
\\
&&\qquad= \lim_{K \to\infty}\sup_{n\ge1} \sum
_{\ell\ge K} {\mathbf P}\bigl(n^{-1/2}
M(T_n,S_k) \ge\ell^{1/p}\bigr)
\\
&&\qquad= 0.
\end{eqnarray*}
This establishes the claimed uniform integrability.

Finally, note that convergence in distribution and the above uniform
integrability imply that in any space in which (a sequence of random
variables with the laws of) $(n^{-1/2}M(T_n,S_k),n \ge1)$ converges
\emph{in probability} to $\chi_k$ (a chi random variable with $2k$
degrees of freedom), we additionally have convergence in $L^p$. This
follows by standard arguments, for example, Theorem 13.7 of \cite{MR1155402}.

\section{Excursions, bridges, trees and forests}\label{secapp}
In this section, we describe the transformations of Section~\ref
{secnovel} in the language of excursions.
This perspective on the results serves two purposes. First, in the
excursion framework, a similarity
is immediately apparent,
between the results of the current paper and results of Aldous and
Pitman \cite{AlPi1994} on scaling limits of random mappings and
on decompositions of reflecting Brownian bridge. Though there seems to
be no direct link between the main
results of the two papers, the idea that they may possess a common
strengthening is intriguing. Second, as noted
in the body of Section~\ref{secnovel}, a careful reader may have had
questions about the precision of the
definitions of some of the random objects under consideration, and the
excursion-theoretic description clarifies
such matters.

Let $e=(e(t),0 \le t \le1)$ be a standard Brownian excursion, and
write $\mathcal{T}_e$ for the $\mathbb{R}$-tree coded by $e$. (We
recall that the
points of $\mathcal{T}_e$ are equivalence classes $\{[x],0 \le x \le
1\}$,
where points $x,y \in[0,1]$ are equivalent if $e(x)=e(y)=\inf\{
e(z)\dvtx
x \le z \le y\}$, and refer the reader to \cite{Legall2005} for more
details of this standard construction.)

Next, let $\mathcal A_e=\{(s,y)\in[0,1]\times\mathbb{R}^+\dvtx0\le
y\le
e(s)\}$
be the set of points lying above the $x$-axis and below the graph of
$e$. For each point $(x,y)$ in $A_e^o$, the interior of $A_e$, let
\[
\underline{s}(x,y)  =\underline{s}(x,y,e) = \inf \bigl\{x'\dvtx
x' \in(0,x), e(z) \ge y\ \forall z \in \bigl[x',x
\bigr] \bigr\}
\]
and let
\[
\bar{s}(x,y)  =\bar{s}(x,y,e) = \sup \bigl\{x'\dvtx
x' \in(x,1), e(z) \ge y\ \forall z \in \bigl[x,x'
\bigr] \bigr\}.
\]
In other words, the line segment
$[\underline s(x,y), \bar{s}(x,y)]\times\{y\}$
is the maximal horizontal line segment through $(x,y)$ contained in
$\mathcal A_e$.

We wish to obtain an excursion-theoretic representation of the Poisson
process on $\operatorname{skel}(\mathcal{T}_e) \times[0,\infty)$
with intensity measure
$\ell\otimes\mathrm{Leb}_{[0,\infty)}$, where $\ell$ is the length
measure on $\operatorname{skel}(\mathcal{T}_e)$ and $\mathrm
{Leb}_{[0,\infty)}$
is Lebesgue measure on
$[0,\infty)$. To do so, for $(x,y)\in\mathcal A_e^o$, we view the
points of $[\underline s(x,y), \bar{s}(x,y))\times\{y\}$ as
representing the point $[\underline s(x,y)]$ of $\operatorname
{skel}(\mathcal{T}_e)$.
We then consider a process $\mathcal{P}_e^{\circ}$ which,
conditional on $e$, is a Poisson process on $\mathcal{A}_e^o\times
[0,\infty)$ with intensity measure at $((x,y),t)$ given by
\[
\frac{\ddd\mathrm{Leb}_{\mathcal{A}_e^o}\otimes\ddd\mathrm
{Leb}_{[0,\infty)}}{\bar{s}(x,y,e)-\underline s(x,y,e)}.
\]
For $t \in[0,\infty)$, let
\[
X_t=X_t \bigl(e,\mathcal{P}_e^{\circ}
\bigr) = \bigl\{z \in[0,1]\dvtx\exists \bigl((x,y),s \bigr) \in\mathcal
{P}_e^{\circ},s \le t, z \in \bigl[\underline{s}(x,y),
\bar{s}(x,y) \bigr] \bigr\}.
\]
In words, the (equivalence classes of) points of $X_t$ are the points
of $\mathcal{T}_e$ lying
in subtrees that have been cut by $\mathcal{P}_e^{\circ}$ by time
$t$. We
define $X_{t-}$ accordingly, let $Y_t=[0,1]\setminus X_t$ and let
$Y_{t-}=[0,1]\setminus X_{t-}$.

Next, for $0 \le t \le\infty$,
let $m_t=\mathrm{Leb}_{[0,1]}(Y_t)$ be the Lebesgue measure of the
points that are not yet cut at time $t$, and
let $m_{t-}=\mathrm{Leb}_{[0,1]}(Y_{t-})$. Then let $\mathcal{P}_e =
\{p
=((x,y),t) \in\mathcal{P}_e^{\circ}\dvtx m_t < m_{t-}\}$
for the set of points that reduce the measure of the ``uncut subtree.''
We next explain how the points of $\mathcal{P}_e$ yield a family of
transformations of the excursion $e$.

For
$z \in\widebar{Y}_t$, the closure of $Y_t$,
let $v_t(z) = \mathrm{Leb}_{[0,1]}([0,z] \cap Y_t)$. The function $v_t
\dvtx\widebar{Y}_t \to[0,m_t]$ is nondecreasing. Furthermore, the
results of \cite{aldous1998standard} imply that $v_t(1)=m_t$ and that
for $0 \le z < z' \le1$ we have $v_t(z)=v_t(z')$ if and only if there
exists $(x,y) \in\mathcal A_e^o$ such that $z= \underline{s}(x,y)$ and
$z'= \bar{s}(x,y)$. In other words, $v_t(z)=v_t(z')$ precisely if
$[z]=[z']$ is the root of a subtree that is cut before or at time $t$.

Let $e_t^0\dvtx[1-m_t,1] \to[0,\infty)$ be given by setting
$e_t^0(z) =
e(v_t^{-1}(z-(1-m_t)))$, where $v_t^{-1}(u)=\inf\{x \dvtx v_t(x)\ge
u\}$
[we could in fact take $v_t^{-1}(u)$ to be any point in the pre-image
of $u$ under $v_t$; the comments of the preceding paragraph show that
the value of $e(v_t^{-1}(u))$ does not depend on this choice]. Then
Theorem~4 of \cite{aldous1998standard}, together with the comments in
Section~3.5 of that paper, implies that conditional on $m_t$, if the
function $e_t^0$ is translated to have domain $[0,m_t]$ then the result
is distributed as a standard Brownian excursion of length $m_t$. We
define $m_{t-}$, $v_{t-}$ and the excursion $e_{t-}$ similarly.

Next, for each point $p=((x,y),t)$ of $\mathcal{P}_e$, we define a random
function $e^p$ with
domain $[1-m_{t-},1-m_t]$ as follows.
For $z \in[1-m_{t-},1-m_t]$, set
\[
e^p(z) = e_{t-} \bigl(v_{t-}^{-1}
\bigl(\underline{s}(x,y) \bigr)+z \bigr).
\]
Notice that
\[
(1-m_t)-(1-m_{t-})=m_{t-}-m_t =
v_{t-} \bigl(\bar{s}(x,y) \bigr)-v_{t-} \bigl(
\underline{s}(x,y) \bigr).
\]
Translated to have range $[0,v_{t-}(\bar{s}(x,y))-v_{t-}(\underline
{s}(x,y))]$, the excursion $e^p$ then codes the tree cut by point $p$
under the standard coding of trees by excursions.

Finally, for $t \in[0,\infty)$ let $e_t\dvtx[0,1] \to[0,\infty)$
be the unique
function such that $e_t|_{[1-m_t,1]}\equiv e_t^0$ and such that
for each $p =(x,y,s) \in\mathcal{P}_e$ with $0 \le s \le t$,
\[
e_t|_{[1-m_{s-},1-m_s]} \equiv e^p.
\]
The function $e_t$ is the ``concatenation'' of the functions
\[
\bigl\{e^p, p=(x,y,s) \in\mathcal{P}_e\dvtx0 \le s \le
t \bigr\}
\]
and of the function $e_t^0$.
We define the function $e_{t-}$ similarly.
The function $e_t$ is comprised of a countably infinite number of
excursions away from zero; the trees coded by these excursions together
comprise the $\mathbb{R}$-forest $(\mathcal{F}_t,d_t,\mu_t)$ of
Section~\ref
{secnovel}. A similar coding of a random continuum forest, by a
reflecting Brownian bridge conditioned on its local time at zero, is
described in \cite{aldous1998standard}, Section~3.5.

The random variables $(e_t,t \ge0)$ are consistent
in the sense that for any fixed $s \in[0,1)$, there is
an almost surely finite time $t_0$ such that for all
$t' > t \ge t_0$, $e_{t'}|_{[0,s]} = e_t|_{[0,s]}$. It follows
that the limit $e_{\infty} = \lim_{t \to\infty} e_t$ is almost surely
well-defined.

In the current terminology, for $0 \le t \le\infty$, we have
\[
L(t) = \int_0^t m_s \,\ddd
s.
\]
We view $e_t|_{[0,1-m_t]}=e_{\infty}|_{[0,1-m_t]}$ as
coding a random measured $\mathbb{R}$-tree with mass $1-m_t$, as follows.
Let $d_t^*\dvtx[0,1-m_t] \to[0,\infty)$ be given by setting
\[
d_t^*(u,v) = e_t(v)+e_t(u) - 2\inf
_{u \le s \le v} e_t(s) + {L(s_v)-L(s_u)},
\]
{for all $0 \le u \le v \le1-m_t$ such that there exist $s_u,s_v\in
[0,t]$ for which
$u\in[1-m_{s_u-},1-m_{s_u})$ and
$v\in[1-m_{s_v-},1-m_{s_v})$.}

Then the tree $(\widehat{\mathcal{T}}_t,\hat d_t,\hat\mu_t)$ of
Section~\ref{secnovel} may be defined as follows.
Set $\widehat{\mathcal{T}}_t = \{[u], 0 \le u \le1-m_t\}$, where $[u]$
denotes the
equivalence class
of $u\dvtx[u]=\{0 \le v \le1-m_t\dvtx d_t^*(u,v)=0\}$.
Let $\hat d_t$ be the push-forward of $d_t^*$ to $\widehat{\mathcal{T}}_t$,
and let
$\hat\mu_t$ be the push-forward of Lebesgue measure on $[0,1-m_t]$ to
$\widehat{\mathcal{T}}_t$.

The content of the first assertion of Theorem~\ref{teoexcbr} is that
$e_{\infty}$ is distributed as a reflecting Brownian bridge; we may see
the equivalence between the first part of Theorem~\ref{teoexcbr} and
the latter statement as follows. First, a standard and trivial
extension of Theorem~\ref{teocayleygh}, states that a uniformly
random doubly-marked tree on $[n]$ converges to $(\mathcal{T},d,\mu,(\rho,\rho
'))$ with respect to $d_{\mathrm{GHP}}^2$, where $(\mathcal
{T},d,\mu)$ is a
Brownian CRT and $\rho,\rho'$ are independent elements of $\mathcal
{T}$, each
with law $\mu$. Next, recall the standard one-to-one map between
doubly-marked trees on $[n]$ and ordered rooted forests on $[n]$ which
``removes the edges on the path between the two marked vertices.''
Finally, results from \cite{AlPi1994}---in particular, the first two
distributional convergence results in Theorem 8 of that paper, together
with the remark in Section~10---imply that that the contour process of
a uniformly random ordered rooted forest on $[n]$ converges after
appropriate rescaling to a reflecting Brownian bridge. (We remark that
a direct encoding of a doubly-rooted Brownian CRT by reflecting
Brownian bridge, also mentioned in the \hyperref[secintro]{Introduction}, is given in \cite
{bertoinpitman94path}. The latter is closely related to, but distinct
from, the encoding obtained by considering ordered rooted forests as above.)

Next, for each point $p=((x,y),t) \in\mathcal{P}_e$, let
\[
u_p =1-m_t + v_t \bigl(\underline s(x,y,e)
\bigr) \in[1-m_t,1].
\]
If we view $e_t^0$ as coding a tree, then the (equivalence class of
the) point $u_p$ is a leaf of this tree.
Then let $y_p=y_p(e,\mathcal{P}_e)$ be the push-forward of $u_p$ under
the map
that sends $e_t \to e_{\infty}$. In other words,
let $p'=((x',y'),t')$ be the a.s. unique point of $\mathcal{P}_e$ with
$t' >
t$, with $\underline s(x',y',e) < \underline s(x,y,e)$, with
$\bar{s}(x',y',e) > \bar{s}(x,y,e)$, and minimizing $t'$
subject to these constraints.
Then we set
\[
y_p(e,\mathcal{P}_e) = 1-m_{t'-} +
v_{t'-} \bigl(\underline s(x,y,e) \bigr)-v_{t'-} \bigl(
\underline s \bigl(x',y',e \bigr) \bigr).
\]
The second assertion of Theorem~\ref{teoexcbr} is that conditional on
$e_{\infty}$, the law of $\{y_p,p \in\mathcal{P}_e\}$
is the same as that of the following family of random variables. Let $Z
= \{z \in(0,1)\dvtx e_{\infty}(z)=0\}$. Then
independently for each $z \in Z$ let $Y_z$ be uniform on $[z,1]$.

We remark that a related family of random variables plays a role in
Theorem~8 of \cite{AlPi1994} (in particular in the third distributional
convergence of that theorem). The latter theorem, which describes a
distributional limit
for uniformly random mappings of $[n]$, has several suggestive
similarities to our main result.
We do not see any direct relation between the distributional limits
described in that paper and those established here.
Establishing such a relation would certainly be of interest, and would
likely yield insights in both the discrete and limiting settings.
\end{appendix}



%

\printaddresses


\begin{thebibliography}{49}

\bibitem{AbDe2013a}
%
\begin{barticle}[mr]
\bauthor{\bsnm{Abraham},~\bfnm{Romain}\binits{R.}} \AND
\bauthor{\bsnm{Delmas},~\bfnm{Jean-Fran{\c{c}}ois}\binits{J.-F.}}
(\byear{2013}).
\btitle{The forest associated with the record process on a {L}\'evy tree}.
\bjournal{Stochastic Process. Appl.}
\bvolume{123}
\bpages{3497--3517}.
\bid{doi={10.1016/j.spa.2013.04.017}, issn={0304-4149}, mr={3071387}}
\end{barticle}
%
\bptok{imsref}%
\endbibitem

\bibitem{AbDe2013b}
%
\begin{barticle}[mr]
\bauthor{\bsnm{Abraham},~\bfnm{Romain}\binits{R.}} \AND
\bauthor{\bsnm{Delmas},~\bfnm{Jean-Fran{\c{c}}ois}\binits{J.-F.}}
(\byear{2013}).
\btitle{Record process on the continuum random tree}.
\bjournal{ALEA Lat. Am. J. Probab. Math. Stat.}
\bvolume{10}
\bpages{225--251}.
\bid{issn={1980-0436}, mr={3083925}}
\end{barticle}
%
\bptok{imsref}%
\endbibitem

\bibitem{webref2}
%
\begin{bmisc}[auto:STB|2014/02/12|14:17:21]
\bauthor{\bsnm{Addario-Berry},~\bfnm{L.}\binits{L.}},
\bauthor{\bsnm{Broutin},~\bfnm{N.}\binits{N.}} \AND
\bauthor{\bsnm{Holmgren},~\bfnm{C.}\binits{C.}}
(\byear{2010}).
\bhowpublished{Cutting down trees with a Markov chainsaw (with online
slides). YEP VII Seminar (March 2010).
\url{http://www.eurandom.tue.nl/events/workshops/2010/YEPVII/YEVIIAbstracts.htm}.}
\end{bmisc}
%
\bptok{imsref}%
\endbibitem

\bibitem{Aldous1991}
%
\begin{bincollection}[mr]
\bauthor{\bsnm{Aldous},~\bfnm{David}\binits{D.}}
(\byear{1991}).
\btitle{The continuum random tree. {II}. {A}n overview}.
In \bbooktitle{Stochastic ({D}urham, 1990)}
(\beditor{\binits{M.~T.} \bsnm{Barlow}}
and
\beditor{\binits{N.~H.} \bsnm{Bingham}}, eds.)
\bpages{23--70}.
\bpublisher{Cambridge Univ. Press},
\blocation{Cambridge}.
\bid{doi={10.1017/CBO9780511662980.003}, mr={1166406}}
\end{bincollection}
%
\bptok{imsref}%
\endbibitem

\bibitem{Aldous1991a}
%
\begin{barticle}[mr]
\bauthor{\bsnm{Aldous},~\bfnm{David}\binits{D.}}
(\byear{1991}).
\btitle{Asymptotic fringe distributions for general families of random trees}.
\bjournal{Ann. Appl. Probab.}
\bvolume{1}
\bpages{228--266}.
\bid{issn={1050-5164}, mr={1102319}}
\end{barticle}
%
\bptok{imsref}%
\endbibitem

\bibitem{Aldous1991b}
%
\begin{barticle}[mr]
\bauthor{\bsnm{Aldous},~\bfnm{David}\binits{D.}}
(\byear{1991}).
\btitle{The continuum random tree. {I}}.
\bjournal{Ann. Probab.}
\bvolume{19}
\bpages{1--28}.
\bid{issn={0091-1798}, mr={1085326}}
\end{barticle}
%
\bptok{imsref}%
\endbibitem

\bibitem{Aldous1993a}
%
\begin{barticle}[mr]
\bauthor{\bsnm{Aldous},~\bfnm{David}\binits{D.}}
(\byear{1993}).
\btitle{The continuum random tree. {III}}.
\bjournal{Ann. Probab.}
\bvolume{21}
\bpages{248--289}.
\bid{issn={0091-1798}, mr={1207226}}
\end{barticle}
%
\bptok{imsref}%
\endbibitem

\bibitem{aldous1998standard}
%
\begin{barticle}[mr]
\bauthor{\bsnm{Aldous},~\bfnm{David}\binits{D.}} \AND
\bauthor{\bsnm{Pitman},~\bfnm{Jim}\binits{J.}}
(\byear{1998}).
\btitle{The standard additive coalescent}.
\bjournal{Ann. Probab.}
\bvolume{26}
\bpages{1703--1726}.
\bid{doi={10.1214/aop/1022855879}, issn={0091-1798}, mr={1675063}}
\end{barticle}
%
\bptok{imsref}%
\endbibitem

\bibitem{AlSt2003}
%
\begin{bincollection}[auto]
\bauthor{\bsnm{Aldous},~\bfnm{David}\binits{D.}} \AND
\bauthor{\bsnm{Steele},~\bfnm{J.~Michael}\binits{J.~M.}}
(\byear{2003}).
\btitle{The objective method: Probabilistic combinatorial optimization
and local weak convergence}.
In \bbooktitle{Discrete and Combinatorial Probability}
(\beditor{\binits{H.} \bsnm{Kesten}}, ed.)
\bpages{1--72}.
\bpublisher{Springer},
\blocation{Berlin}.
\bptnote{check year}%
\end{bincollection}
%
\bptok{imsref}%
\endbibitem

\bibitem{Aldous1990}
%
\begin{barticle}[mr]
\bauthor{\bsnm{Aldous},~\bfnm{David~J.}\binits{D.~J.}}
(\byear{1990}).
\btitle{The random walk construction of uniform spanning trees and
uniform labelled trees}.
\bjournal{SIAM J. Discrete Math.}
\bvolume{3}
\bpages{450--465}.
\bid{doi={10.1137/0403039}, issn={0895-4801}, mr={1069105}}
\end{barticle}
%
\bptok{imsref}%
\endbibitem

\bibitem{AlPi1994}
%
\begin{barticle}[mr]
\bauthor{\bsnm{Aldous},~\bfnm{David~J.}\binits{D.~J.}} \AND
\bauthor{\bsnm{Pitman},~\bfnm{Jim}\binits{J.}}
(\byear{1994}).
\btitle{Brownian bridge asymptotics for random mappings}.
\bjournal{Random Structures Algorithms}
\bvolume{5}
\bpages{487--512}.
\bid{doi={10.1002/rsa.3240050402}, issn={1042-9832}, mr={1293075}}
\end{barticle}
%
\bptok{imsref}%
\endbibitem

\bibitem{Bertoin2006}
%
\begin{bbook}[mr]
\bauthor{\bsnm{Bertoin},~\bfnm{Jean}\binits{J.}}
(\byear{2006}).
\btitle{Random Fragmentation and Coagulation Processes}.
\bpublisher{Cambridge Univ. Press},
\blocation{Cambridge}.
\bid{doi={10.1017/CBO9780511617768}, mr={2253162}}
\end{bbook}
%
\bptok{imsref}%
\endbibitem

\bibitem{Bertoin2012a}
%
\begin{bincollection}[auto]
\bauthor{\bsnm{Bertoin},~\bfnm{Jean}\binits{J.}}
(\byear{2012}).
\btitle{Fires on trees}.
In \bbooktitle{Annales de l'Institut Henri Poincar{\'e}, Probabilit{\'
e}s et Statistiques}
\bvolume{48}
\bpages{909--921}.
\bid{doi={10.1214/11-AIHP435}, issn={0246-0203}, mr={3052398}}
\end{bincollection}
%
\bptok{imsref}%
\endbibitem

\bibitem{BeMi2012a}
%
\begin{barticle}[mr]
\bauthor{\bsnm{Bertoin},~\bfnm{Jean}\binits{J.}} \AND
\bauthor{\bsnm{Miermont},~\bfnm{Gr{\'e}gory}\binits{G.}}
(\byear{2013}).
\btitle{The cut-tree of large {G}alton--{W}atson trees and the
{B}rownian {CRT}}.
\bjournal{Ann. Appl. Probab.}
\bvolume{23}
\bpages{1469--1493}.
\bid{issn={1050-5164}, mr={3098439}}
\end{barticle}
%
\bptok{imsref}%
\endbibitem

\bibitem{bertoinpitman94path}
%
\begin{barticle}[mr]
\bauthor{\bsnm{Bertoin},~\bfnm{Jean}\binits{J.}} \AND
\bauthor{\bsnm{Pitman},~\bfnm{Jim}\binits{J.}}
(\byear{1994}).
\btitle{Path transformations connecting {B}rownian bridge, excursion
and meander}.
\bjournal{Bull. Sci. Math.}
\bvolume{118}
\bpages{147--166}.
\bid{issn={0007-4497}, mr={1268525}}
\end{barticle}
%
\bptok{imsref}%
\endbibitem

\bibitem{billingsley1968cpm}
%
\begin{bbook}[mr]
\bauthor{\bsnm{Billingsley},~\bfnm{Patrick}\binits{P.}}
(\byear{1968}).
\btitle{Convergence of Probability Measures}.
\bpublisher{Wiley},
\blocation{New York}.
\bid{mr={0233396}}
\end{bbook}
%
\bptok{imsref}%
\endbibitem

\bibitem{Broder1989}
%
\begin{bincollection}[auto:STB|2014/02/12|14:17:21]
\bauthor{\bsnm{Broder},~\bfnm{A.}\binits{A.}}
(\byear{1989}).
\btitle{Generating random spanning trees}.
In \bbooktitle{30th Annual Symposium on Foundations of Computer Science}
\bpages{442--447}.
\bpublisher{IEEE},
\blocation{New York}.
\end{bincollection}
%
\bptok{imsref}%
\endbibitem

\bibitem{burago01}
%
\begin{bbook}[mr]
\bauthor{\bsnm{Burago},~\bfnm{Dmitri}\binits{D.}},
\bauthor{\bsnm{Burago},~\bfnm{Yuri}\binits{Y.}} \AND
\bauthor{\bsnm{Ivanov},~\bfnm{Sergei}\binits{S.}}
(\byear{2001}).
\btitle{A Course in Metric Geometry}.
\bseries{Graduate Studies in Mathematics}
\bvolume{33}.
\bpublisher{Amer. Math. Soc.},
\blocation{Providence, RI}.
\bid{mr={1835418}}
\end{bbook}
%
\bptok{imsref}%
\endbibitem

\bibitem{ChMa2009}
%
\begin{bmisc}[auto:STB|2014/02/12|14:17:21]
\bauthor{\bsnm{Chassaing},~\bfnm{P.}\binits{P.}} \AND
\bauthor{\bsnm{Marchand},~\bfnm{R.}\binits{R.}}
(\byear{2009}).
\bhowpublished{Personal communication}.
\end{bmisc}
%
\bptok{imsref}%
\endbibitem

\bibitem{DaVe2007}
%
\begin{bbook}[auto]
\bauthor{\bsnm{Daley},~\bfnm{D.~J.}\binits{D.~J.}} \AND
\bauthor{\bsnm{Vere-Jones},~\bfnm{D.}\binits{D.}}
(\byear{2007}).
\btitle{An Introduction to the Theory of Point Processes. {V}ol. {II}:
General Theory and Structure}.
\bpublisher{Springer},
\blocation{New York}.
\end{bbook}
%
\bptok{imsref}%
\endbibitem

\bibitem{DeJa2011a}
%
\begin{barticle}[mr]
\bauthor{\bsnm{Devroye},~\bfnm{Luc}\binits{L.}} \AND
\bauthor{\bsnm{Janson},~\bfnm{Svante}\binits{S.}}
(\byear{2011}).
\btitle{Distances between pairs of vertices and vertical profile in
conditioned {G}alton--{W}atson trees}.
\bjournal{Random Structures Algorithms}
\bvolume{38}
\bpages{381--395}.
\bid{doi={10.1002/rsa.20319}, issn={1042-9832}, mr={2829308}}
\end{barticle}
%
\bptok{imsref}%
\endbibitem

\bibitem{DrIkMoRo2009}
%
\begin{barticle}[mr]
\bauthor{\bsnm{Drmota},~\bfnm{Michael}\binits{M.}},
\bauthor{\bsnm{Iksanov},~\bfnm{Alex}\binits{A.}},
\bauthor{\bsnm{Moehle},~\bfnm{Martin}\binits{M.}} \AND
\bauthor{\bsnm{Roesler},~\bfnm{Uwe}\binits{U.}}
(\byear{2009}).
\btitle{A limiting distribution for the number of cuts needed to
isolate the root of a random recursive tree}.
\bjournal{Random Structures Algorithms}
\bvolume{34}
\bpages{319--336}.
\bid{doi={10.1002/rsa.20233}, issn={1042-9832}, mr={2504401}}
\end{barticle}
%
\bptok{imsref}%
\endbibitem

\bibitem{FiKaPa2006}
%
\begin{barticle}[mr]
\bauthor{\bsnm{Fill},~\bfnm{James~Allen}\binits{J.~A.}},
\bauthor{\bsnm{Kapur},~\bfnm{Nevin}\binits{N.}} \AND
\bauthor{\bsnm{Panholzer},~\bfnm{Alois}\binits{A.}}
(\byear{2006}).
\btitle{Destruction of very simple trees}.
\bjournal{Algorithmica}
\bvolume{46}
\bpages{345--366}.
\bid{doi={10.1007/s00453-006-0100-1}, issn={0178-4617}, mr={2291960}}
\end{barticle}
%
\bptok{imsref}%
\endbibitem

\bibitem{HaMi2012a}
%
\begin{barticle}[mr]
\bauthor{\bsnm{Haas},~\bfnm{B{\'e}n{\'e}dicte}\binits{B.}} \AND
\bauthor{\bsnm{Miermont},~\bfnm{Gr{\'e}gory}\binits{G.}}
(\byear{2012}).
\btitle{Scaling limits of {M}arkov branching trees with applications
to {G}alton--{W}atson and random unordered trees}.
\bjournal{Ann. Probab.}
\bvolume{40}
\bpages{2589--2666}.
\bid{doi={10.1214/11-AOP686}, issn={0091-1798}, mr={3050512}}
\end{barticle}
%
\bptok{imsref}%
\endbibitem

\bibitem{Holmgren2008}
%
\begin{bincollection}[mr]
\bauthor{\bsnm{Holmgren},~\bfnm{Cecilia}\binits{C.}}
(\byear{2008}).
\btitle{Random records and cuttings in split trees: Extended abstract}.
In \bbooktitle{Fifth {C}olloquium on {M}athematics and {C}omputer {S}cience}.
\bpages{269--281}.
\bpublisher{Assoc. Discrete Math. Theor. Comput. Sci.}, \blocation{Nancy}.
\bid{mr={2508793}}
\end{bincollection}
%
\bptok{imsref}%
\endbibitem

\bibitem{IkMo2007}
%
\begin{barticle}[mr]
\bauthor{\bsnm{Iksanov},~\bfnm{Alex}\binits{A.}} \AND
\bauthor{\bsnm{M{\"o}hle},~\bfnm{Martin}\binits{M.}}
(\byear{2007}).
\btitle{A probabilistic proof of a weak limit law for the number of
cuts needed to isolate the root of a random recursive tree}.
\bjournal{Electron. Commun. Probab.}
\bvolume{12}
\bpages{28--35}.
\bid{doi={10.1214/ECP.v12-1253}, issn={1083-589X}, mr={2407414}}
\end{barticle}
%
\bptok{imsref}%
\endbibitem

\bibitem{Janson2004}
%
\begin{bincollection}[mr]
\bauthor{\bsnm{Janson},~\bfnm{Svante}\binits{S.}}
(\byear{2004}).
\btitle{Random records and cuttings in complete binary trees}.
In \bbooktitle{Mathematics and Computer Science. {III}: Algorithms,
Trees, Combinatorics and Probability (Vienna)}
(\beditor{\binits{M.} \bsnm{Drmota}},
\beditor{\binits{P.} \bsnm{Flajolet}},
\beditor{\binits{D.} \bsnm{Gardy}}
and
\beditor{\binits{B.} \bsnm{Gittenberger}}, eds.)
\bpages{241--253}.
\bpublisher{Birkh\"auser},
\blocation{Basel}.
\bid{mr={2090513}}
\end{bincollection}
%
\bptok{imsref}%
\endbibitem

\bibitem{Janson2006c}
%
\begin{barticle}[mr]
\bauthor{\bsnm{Janson},~\bfnm{Svante}\binits{S.}}
(\byear{2006}).
\btitle{Random cutting and records in deterministic and random trees}.
\bjournal{Random Structures Algorithms}
\bvolume{29}
\bpages{139--179}.
\bid{doi={10.1002/rsa.20086}, issn={1042-9832}, mr={2245498}}
\end{barticle}
%
\bptok{imsref}%
\endbibitem

\bibitem{Kennedy1976}
%
\begin{barticle}[mr]
\bauthor{\bsnm{Kennedy},~\bfnm{Douglas~P.}\binits{D.~P.}}
(\byear{1976}).
\btitle{The distribution of the maximum {B}rownian excursion}.
\bjournal{J. Appl. Probab.}
\bvolume{13}
\bpages{371--376}.
\bid{issn={0021-9002}, mr={0402955}}
\end{barticle}
%
\bptok{imsref}%
\endbibitem

\bibitem{kesten86sub}
%
\begin{barticle}[mr]
\bauthor{\bsnm{Kesten},~\bfnm{Harry}\binits{H.}}
(\byear{1986}).
\btitle{Subdiffusive behavior of random walk on a random cluster}.
\bjournal{Ann. Inst. Henri Poincar\'e Probab. Stat.}
\bvolume{22}
\bpages{425--487}.
\bid{issn={0246-0203}, mr={0871905}}
\end{barticle}
%
\bptok{imsref}%
\endbibitem

\bibitem{Kolchin1986}
%
\begin{bbook}[mr]
\bauthor{\bsnm{Kolchin},~\bfnm{Valentin~F.}\binits{V.~F.}}
(\byear{1986}).
\btitle{Random Mappings}.
\bpublisher{Optimization Software Inc. Publications Division},
\blocation{New York}.
\bid{mr={0865130}}
\end{bbook}
%
\bptok{imsref}%
\endbibitem

\bibitem{KuPa2008}
%
\begin{barticle}[mr]
\bauthor{\bsnm{Kuba},~\bfnm{Markus}\binits{M.}} \AND
\bauthor{\bsnm{Panholzer},~\bfnm{Alois}\binits{A.}}
(\byear{2008}).
\btitle{Isolating a leaf in rooted trees via random cuttings}.
\bjournal{Ann. Comb.}
\bvolume{12}
\bpages{81--99}.
\bid{doi={10.1007/s00026-008-0338-1}, issn={0218-0006}, mr={2401138}}
\end{barticle}
%
\bptok{imsref}%
\endbibitem

\bibitem{KuPa2008a}
%
\begin{barticle}[mr]
\bauthor{\bsnm{Kuba},~\bfnm{Markus}\binits{M.}} \AND
\bauthor{\bsnm{Panholzer},~\bfnm{Alois}\binits{A.}}
(\byear{2008}).
\btitle{Isolating nodes in recursive trees}.
\bjournal{Aequationes Math.}
\bvolume{76}
\bpages{258--280}.
\bid{doi={10.1007/s00010-008-2929-7}, issn={0001-9054}, mr={2461893}}
\end{barticle}
%
\bptok{imsref}%
\endbibitem

\bibitem{Legall2005}
%
\begin{barticle}[mr]
\bauthor{\bparticle{Le}~\bsnm{Gall},~\bfnm{Jean-Fran{\c
{c}}ois}\binits{J.-F.}}
(\byear{2005}).
\btitle{Random trees and applications}.
\bjournal{Probab. Surv.}
\bvolume{2}
\bpages{245--311}.
\bid{doi={10.1214/154957805100000140}, issn={1549-5787}, mr={2203728}}
\end{barticle}
%
\bptok{imsref}%
\endbibitem

\bibitem{legall06realtrees}
%
\begin{barticle}[mr]
\bauthor{\bparticle{Le} \bsnm{Gall},~\bfnm{Jean-Fran{\c
{c}}ois}\binits{J.-F.}}
(\byear{2006}).
\btitle{Random real trees}.
\bjournal{Ann. Fac. Sci. Toulouse Math. (6)}
\bvolume{15}
\bpages{35--62}.
\bid{issn={0240-2963}, mr={2225746}}
\end{barticle}
%
\bptok{imsref}%
\endbibitem

\bibitem{lp}
%
\begin{bmisc}[auto:STB|2014/02/12|14:17:21]
\bauthor{\bsnm{Lyons},~\bfnm{Russell}\binits{R.}} \AND
\bauthor{\bsnm{Peres},~\bfnm{Yuval}\binits{Y.}}
(\byear{2012}).
\bhowpublished{Probability on trees and networks. Unpublished manuscript.}
\end{bmisc}
%
\bptok{imsref}%
\endbibitem

\bibitem{mcd}
%
\begin{bincollection}[mr]
\bauthor{\bsnm{McDiarmid},~\bfnm{Colin}\binits{C.}}
(\byear{1998}).
\btitle{Concentration}.
In \bbooktitle{Probabilistic Methods for Algorithmic Discrete Mathematics}
(\beditor{\binits{M.} \bsnm{Habib}},
\beditor{\binits{C.} \bsnm{McDiarmid}},
\beditor{\binits{J.} \bsnm{Ramirez-Alfonsin}}
and
\beditor{\binits{B.} \bsnm{Reed}}, eds.)
\bpages{195--248}.
\bpublisher{Springer},
\blocation{Berlin}.
\bid{mr={1678578}}
\end{bincollection}
%
\bptok{imsref}%
\endbibitem

\bibitem{MR0263685}
%
\begin{barticle}[mr]
\bauthor{\bsnm{Meir},~\bfnm{A.}\binits{A.}} \AND
\bauthor{\bsnm{Moon},~\bfnm{J.~W.}\binits{J.~W.}}
(\byear{1970}).
\btitle{The distance between points in random trees}.
\bjournal{J. Combin. Theory}
\bvolume{8}
\bpages{99--103}.
\bid{mr={0263685}}
\end{barticle}
%
\bptok{imsref}%
\endbibitem

\bibitem{MeMo1970}
%
\begin{barticle}[mr]
\bauthor{\bsnm{Meir},~\bfnm{A.}\binits{A.}} \AND
\bauthor{\bsnm{Moon},~\bfnm{J.~W.}\binits{J.~W.}}
(\byear{1970}).
\btitle{Cutting down random trees}.
\bjournal{J. Aust. Math. Soc.}
\bvolume{11}
\bpages{313--324}.
\bid{issn={0263-6115}, mr={0284370}}
\end{barticle}
%
\bptok{imsref}%
\endbibitem

\bibitem{MeMo1974}
%
\begin{barticle}[auto:STB|2014/02/12|14:17:21]
\bauthor{\bsnm{Meir},~\bfnm{A.}\binits{A.}} \AND
\bauthor{\bsnm{Moon},~\bfnm{J.~W.}\binits{J.~W.}}
(\byear{1974}).
\btitle{Cutting down recursive trees}.
\bjournal{Math. Biosci.}
\bvolume{21}
\bpages{173--181}.
\end{barticle}
%
\bptok{imsref}%
\endbibitem

\bibitem{MeMo1978}
%
\begin{barticle}[mr]
\bauthor{\bsnm{Meir},~\bfnm{A.}\binits{A.}} \AND
\bauthor{\bsnm{Moon},~\bfnm{J.~W.}\binits{J.~W.}}
(\byear{1978}).
\btitle{On the altitude of nodes in random trees}.
\bjournal{Canad. J. Math.}
\bvolume{30}
\bpages{997--1015}.
\bid{doi={10.4153/CJM-1978-085-0}, issn={0008-414X}, mr={0506256}}
\end{barticle}
%
\bptok{imsref}%
\endbibitem

\bibitem{miermont2009tessellations}
%
\begin{barticle}[mr]
\bauthor{\bsnm{Miermont},~\bfnm{Gr{\'e}gory}\binits{G.}}
(\byear{2009}).
\btitle{Tessellations of random maps of arbitrary genus}.
\bjournal{Ann. Sci. \'Ec. Norm. Sup\'er. (4)}
\bvolume{42}
\bpages{725--781}.
\bid{issn={0012-9593}, mr={2571957}}
\end{barticle}
%
\bptok{imsref}%
\endbibitem

\bibitem{Panholzer2003}
%
\begin{bincollection}[mr]
\bauthor{\bsnm{Panholzer},~\bfnm{Alois}\binits{A.}}
(\byear{2003}).
\btitle{Noncrossing trees revisited: Cutting down and spanning subtrees}.
In \bbooktitle{Discrete Random Walks ({P}aris, 2003)}
\bpages{265--276 (electronic)}.
\bpublisher{Assoc. Discrete Math. Theor. Comput. Sci.},
\blocation{Nancy}.
\bid{mr={2042393}}
\end{bincollection}
%
\bptok{imsref}%
\endbibitem

\bibitem{Panholzer2006}
%
\begin{barticle}[mr]
\bauthor{\bsnm{Panholzer},~\bfnm{Alois}\binits{A.}}
(\byear{2006}).
\btitle{Cutting down very simple trees}.
\bjournal{Quaest. Math.}
\bvolume{29}
\bpages{211--227}.
\bid{doi={10.2989/16073600609486160}, issn={1607-3606}, mr={2233368}}
\end{barticle}
%
\bptok{imsref}%
\endbibitem

\bibitem{Pitman2006}
%
\begin{bbook}[mr]
\bauthor{\bsnm{Pitman},~\bfnm{J.}\binits{J.}}
(\byear{2006}).
\btitle{Combinatorial Stochastic Processes}.
\bseries{Lecture Notes in Math.}
\bvolume{1875}.
\bpublisher{Springer},
\blocation{Berlin}.
\bid{mr={2245368}}
\end{bbook}
%
\bptok{imsref}%
\endbibitem

\bibitem{MR0228386}
%
\begin{barticle}[mr]
\bauthor{\bsnm{Riordan},~\bfnm{John}\binits{J.}}
(\byear{1968}).
\btitle{Forests of labeled trees}.
\bjournal{J. Combin. Theory}
\bvolume{5}
\bpages{90--103}.
\bid{mr={0228386}}
\end{barticle}
%
\bptok{imsref}%
\endbibitem

\bibitem{stroock}
%
\begin{bbook}[mr]
\bauthor{\bsnm{Stroock},~\bfnm{Daniel~W.}\binits{D.~W.}}
(\byear{2011}).
\btitle{Probability Theory. An Analytic View},
\bedition{2nd} ed.
\bpublisher{Cambridge Univ. Press},
\blocation{Cambridge}.
\bid{mr={2760872}}
\end{bbook}
%
\bptok{imsref}%
\endbibitem

\bibitem{villani2009optimal}
%
\begin{bbook}[mr]
\bauthor{\bsnm{Villani},~\bfnm{C{\'e}dric}\binits{C.}}
(\byear{2009}).
\btitle{Optimal Transport: Old and New}.
\bpublisher{Springer},
\blocation{Berlin}.
\bid{doi={10.1007/978-3-540-71050-9}, mr={2459454}}
\end{bbook}
%
\bptok{imsref}%
\endbibitem

\bibitem{MR1155402}
%
\begin{bbook}[mr]
\bauthor{\bsnm{Williams},~\bfnm{David}\binits{D.}}
(\byear{1991}).
\btitle{Probability with Martingales}.
\bpublisher{Cambridge Univ. Press},
\blocation{Cambridge}.
\bid{mr={1155402}}
\end{bbook}
%
\bptok{imsref}%
\endbibitem

\end{thebibliography}
\end{document}